\documentclass{amsart}
\usepackage{amsmath}
\usepackage{graphicx}
\usepackage{hyperref}
\hypersetup{colorlinks=true,citecolor=blue,linkcolor=cyan}
\newtheorem{thm}{Theorem}[section]
\newtheorem{cor}[thm]{Corollary}
\newtheorem{lem}[thm]{Lemma}
\newtheorem{prop}[thm]{Proposition}
\theoremstyle{definition}

\theoremstyle{remark}

\numberwithin{equation}{section}
\newcommand{\F}{\mathcal{F}}
\newcommand{\Ff}{\mathbb{F}}
\newcommand{\R}{\mathcal{R}}
\newcommand{\s}{\mathcal{S}}
\newcommand{\Gal}{\mbox{Gal}}
\newcommand{\Z}{\mathbb{Z}}
\newcommand{\Hh}{\mathcal{H}}

\newcommand{\Prob}{\mbox{Prob}}
\newcommand{\Pp}{\mathbb{P}}
\begin{document}

\title[Distribution of Points on Cyclic Curves Over Finite Fields]{Distribution of Points on Cyclic Curves Over Finite Fields}%
\author{Patrick Meisner}%
%\address{}%
%\email{p_meisner@hotmail.com}%

%\thanks{}%
%\subjclass{}%
%\keywords{}%

%\date{}%
%\dedicatory{}%
%\commby{}%
% ----------------------------------------------------------------
\begin{abstract}

We determine in this paper the distribution of the number of points on the cyclic covers of $\Pp^1(\Ff_q)$ with affine models $C: Y^r = F(X)$, where $F(X) \in \Ff_q[X]$ and $r^{th}$-power free when $q$ is fixed and the genus, $g$, tends to infinity. This generalize the work of Kurlberg and Rudnick and Bucur, David, Feigon and Lalin who considered different families of curves over $\Ff_q$. In all cases, the distribution is given by a sum of random variables.

\end{abstract}
\maketitle
% ----------------------------------------------------------------
\section{Introduction}\label{intro}

For any smooth projective, $C$, of genus $g$ over a finite field know that
$$\#C(\Pp^1(\Ff_q)) = q+1 - \sum_{j=1}^{2g} \alpha_j(C)$$
where the zeta function of $C$ is
$$Z_C(u) = \frac{\prod_{j=1}^{2g} (1-u\alpha_j(C))}{(1-q)(1-qu)},$$
and $|\alpha_j(C)| = q^{\frac{1}{2}}$ for $j=1,\dots,2g$.

The distribution of $\#C(\Pp^1(\Ff_q))$ when $C$ varies over families of curves over $\Ff_q$ is a classical object of study. For several families of curves over $\Ff_q$ Katz and Sarnak \cite{KS} showed that when the genus is fixed and $q$ tends to infinity,
$$\frac{\sum_{j=1}^{2g} \alpha_j(C)}{\sqrt{q}}$$
is distributed as the trace of a random matrix in the monodromy group of the family.

The distribution of the number points on families of curves over finite fields with $q$ fixed while the genus tends to infinity has been a topic of much research recently. It began with Kurlberg and Rudnick \cite{KR} determining the distribution of the number of points of hyperelliptic curves. Hyperelliptic curves are in one-to-one correspondence with Galois extensions of $\Ff_q(X)$ with Galois group $\Z/2\Z$. Bucur, David, Feigon and Lalin \cite{BDFL1},\cite{BDFL2} extended this result to smooth projective curves that are in one-to-one correspondence with Galois extensions of $\Ff_q(X)$ with Galois group $\Z/p\Z$, where $p$ is a prime such that $q\equiv1 \mod{p}$. Recently Lorenzo, Milione and Meleleo \cite{LMM} determined the case for Galois group $(\Z/2\Z)^n$. In this paper we determine the case for cyclic Galois groups $\Z/r\Z$ for any $q\equiv1 \mod{r}$ where $r$ is not necessarily a prime.

Let $K=\Ff_q(X)$ and let $L$ be a finite Galois extension of $K$. Let $r$ be an integer such that $q\equiv 1 \mod{r}$. Suppose that $\Gal(L/K)=\Z/r\Z$. Then there exists a unique smooth projective curve over $\Ff_q$, $C$, such that $L\cong K(C)$. Further, $C$ will have an affine model of the form
$$Y^r = \alpha F(X), \quad \quad F\in \F_{(d_1,\dots,d_{r-1})} \subset \Ff_q[X], \alpha \in \Ff^*_q$$
where
\begin{align*}
\F_{(d_1,\dots,d_{r-1})} = \{&F=f_1f_2^2\cdots f_{r-1}^{r-1}: f_i \in \Ff_q[X] \mbox{ are monic, square-free, pairwise coprime,}\\
& \mbox{and $\deg F_i = d_i$ for $1\leq i \leq r-1$}\}.
\end{align*}

The Riemann-Hurwitz formula (Theorem 7.16 of \cite{rose}) tells us that if we let $d=\sum_{i=1}^{r-1} id_i$, then the genus $g$ of the curve $C$ is given by
$$2g+2r-2 = \sum_{i=1}^{r-1} (r-(r,i))d_i + (r-(r,d))$$
where $(r,i) = \gcd(r,i)$.
Define now the following sets
\begin{align*}
\F^j_{(d_1,\dots,d_{r-1})} & = \F_{(d_1,\dots,d_j-1,\dots,d_{r-1})}\\
\hat{\F}_{(d_1,\dots,d_{r-1})} & = \{\alpha F: F \in \F_{(d_1,\dots,d_{r-1})}, \alpha \in\Ff^*_q\}\\
\hat{\F}^j_{(d_1,\dots,d_{r-1})} & = \hat{\F}_{(d_1,\dots,d_j-1,\dots,d_{r-1})}\\
\F_{[d_1,\dots,d_{r-1}]} & = \F_{(d_1,\dots,d_{r-1})}\cup \bigcup_{j=1}^{r-1} \F^j_{(d_1,\dots,d_{r-1})}\\
\hat{\F}_{[d_1,\dots,d_{r-1}]} & = \hat{\F}_{(d_1,\dots,d_{r-1})}\cup \bigcup_{j=1}^{r-1} \hat{\F}^j_{(d_1,\dots,d_{r-1})}.
\end{align*}

If we now restrict the $d_i$ such that $\sum_{i=1}^{r-1}id_i\equiv0\mod{r}$, then the genus will be invariant over curves with affine models $Y^r=F(x)$ with $F(x)\in\hat{\F}_{[d_1,\dots,d_{r-1}]}$. That is, if we let $\Hh_{g,r}$ be the set of curves of genus $g$ such that $\Gal(L/K)=\Z/r\Z$, we can write
$$\Hh_{g,r} = \bigcup_{\substack{\sum_{i=1}^{r-1}id_i\equiv 0 \mod{r} \\ \sum_{i=1}^{r-1} (r-(r,i))d_i = 2g+2r-2}} \Hh^{(d_1,\dots,d_{r-1})}$$
where $\Hh^{(d_1,\dots,d_{r-1})}$ are the curves with affine models in $\hat{\F}_{[d_1,\dots,d_{r-1}]}$. We will discuss the distribution of points not on the whole of $\Hh_{g,r}$ but on each $\Hh^{(d_1,\dots,d_{r-1})}$.

Kurlberg and Rudnick \cite{KR} first investigated the distribution of points for hyper-elliptic curves ($r=2$). Bucur, David, Feigon and Lalin \cite{BDFL1},\cite{BDFL2} then extended this result to the case where $r=p$, a prime. They noted that the number of points on such a curve will be given by the formula
$$\#C(\Pp^1(\Ff_q)) = \sum_{x\in\Pp^1(\Ff_q)} \sum_{i=0}^{p-1} \chi^i_p(F(x)) = q+1 + \sum_{x\in\Pp^1(\Ff_q)} \sum_{i=1}^{p-1} \chi^i_p(F(x)),$$
where $\chi_p$ is a primitive character on $\Ff_q$ of order $p$. Such a character exists since $q\equiv 1 \mod{p}$. This will be determined by the value $S_p(F) = \sum_{x\in\Pp^1(\Ff_q)} \chi_p(F(x))$ which leads to the result

\begin{thm}[Theorem 1.1 from \cite{BDFL1}]
 If $q$ is fixed and $d_1,\dots,d_{p-1}\to\infty$ then for any $t \in \Z[\zeta_p]$
$$\frac{|\{F \in \hat{\F}_{[d_1,\dots,d_{r-1}]} : S_p(F)=t\}|}{|\hat{\F}_{[d_1,\dots,d_{r-1}]}|} \sim \Prob\left(\sum_{i=1}^{q+1}X_i = t\right)$$
where the $X_i$ are i.i.d random variables such that
$$X_i = \begin{cases} 0 & \mbox{with probability } \frac{p-1}{q+p-1} \\ \zeta_p^j & \mbox{with probability } \frac{q}{p(q+p-1)}  \end{cases}$$
\end{thm}

For general $r$, we prove in Section \ref{Numpts} a formula for the number of point on the curve with affine model $Y^r = F(X)$. In order to state this formula and our main result, we need some notation.

For $d|r$, define
$$F_{(d)}(X) = \prod_{i=1}^{d-1}\left(\prod_{j=0}^{\frac{r}{d}-1} f_{jd+i}(X)\right)^{i} = \prod_{i=1}^{r-1} f_i(X)^{i \mod{d}}.$$
Notice that we could write an affine model for our curve as $Y^r = F_{(r)}(X)$. Further, the $F_{(d)}(X)$ correspond to the subfield extension of $L$. That is if we have $K \subset L' \subset L$, then $L' = K(C_d)$, where $C_d$ is a curve with affine model $Y^d = F_{(d)}(X)$ for some $d|r$. Then Lemma \ref{numpts1} shows that
$$\#C(\Pp^1(\Ff_q))=q+1 + \sum_{d|r}\sum_{\substack{i=1\\(i,d)=1}}^{d-1}  \sum_{x\in\Pp^1(\Ff_q)} \chi_{d}^i(F_{(d)}(x))$$
where $\chi_d$ is a primitive character on $\Ff_q$ of order $d$. Again, such a character exists since $d|r$ and we are assuming that $q\equiv 1 \mod{r}$.
Hence if we define
$$S_d(F_{(d)}) =  \sum_{x\in\Pp^1(\Ff_q)} \chi_{d}(F_{(d)}(x))$$
then the number of points on the curve will be determined by the $S_d(F_{(d)})$ for all $d|r$. This leads us to the main theorem
\begin{thm}\label{mainthm}

Write $r=\prod_{j=1}^n p_j^{t_j}$. If $q$ is fixed, then as $d_1,\dots,d_{r-1}\to\infty$ for any $M_d \in \mathbb{Z}[\zeta_d]$,
$$\frac{|\{F\in\hat{\F}_{[d_1,\dots,d_{r-1}]} : S_d(F_{(d)})) = M_d, \forall d|r, d\not=1 \}|}{|\hat{\F}_{[d_1,\dots,d_{r-1}]}|} \sim \Prob\left(\sum_{i=1}^{q+1} X_{d,i} = M_d, \forall d|r, d\not=1 \right)$$
where $X_{d,i}$ are random variables taking values in $\mu_d \cup \{0\}$ such that
$$\Prob(X_{d,i}=0) = \frac{r-\frac{r}{d}}{q+r-1}$$
$$\Prob(X_{d,i} = \epsilon_{d,i}\not=0) = \frac{q+\frac{r}{d}-1} {d(q+r-1)}.$$
Moreover, if $i\not=j$ then $X_{d,i}$ and $X_{d',j}$ are independent for all $d,d'|r$. However, if we fix $i$, then for all $d|r$
$$X_{d,i} = \prod_{p|d} (X_{p^{v_p(d)},i})^{\sigma_p} \mbox{ where } 1 \leq \sigma_p \leq \frac{d}{p^{v_p(d)}} \mbox{ such that } \sigma_p \equiv (p^{v_p(d)})^{-1} \mod{\frac{d}{p^{v_p(d)}}}.$$
Further, for all $p|r$ and $1 < s \leq v_p(r)$
$$\Prob\left(X_{p^s,i}=0 | X_{p^{s-1},i}=0\right)=1$$
$$\Prob\left(X_{p^{s-1},i} = \epsilon_{p^s,i}^p | X_{p^s,i} = \epsilon_{p^s,i} \right) = 1.$$
Finally, if $d|r$ but $d\not=r$ then
$$\Prob\left(X_{p^s,i} = \epsilon_{p^s,i}\not=0,1 \leq s \leq v_p(d) \mbox{ and } X_{p^s,i} =0, v_p(d)<s\leq v_p(r) \mbox{ for all } p|r\right)$$
$$ = \begin{cases}\frac{\phi(\frac{r}{d})}{d(q+r-1)} & \mbox{if }  \epsilon_{p^{s-1},i} = \epsilon_{p^s,i}^p \mbox{ for all } p|r, 1\leq s \leq v_p(d)  \\ 0  & \mbox{otherwise} \end{cases}$$
and, if $d=r$,
$$\Prob\left(X_{p^s,i} = \epsilon_{p^s,i}, s \leq v_p(r), \mbox{ for all } p|r\right) = \begin{cases} \frac{q }{ r(q+r-1)} & \mbox{if }  \epsilon_{p^{s-1},i} = \epsilon_{p^s,i}^p, 1\leq  s \leq v_p(r), \mbox{ for all } p|r \\ 0 & \mbox{otherwise} \end{cases}.$$

\end{thm}

To simplify notation, we will prove this for $r=p^n$ in full detail in Section \ref{primepower}, and present in Section \ref{generalr} the proof for general $r$. Section \ref{heuristic} gives a heuristic which corroborates the results of Theorem \ref{mainthm}.

\section{Number of Points on the Curve}\label{Numpts}

Let $C$ be the smooth projective curve over $\mathbb{F}_q$ associated to the field $L=\mathbb{F}_q(X)\left(\sqrt[r]{\alpha f_1(X)f_2^2(X)\dots f_{r-1}^{r-1}(X)}\right)$ where the $f_i$ are monic, squarefree and pairwise coprime and $\alpha \in \Ff^*_q$. Let us determine the smooth affine models of these curves. We first consider the affine model
$$Y^r = \alpha f_1(X)f_2^2(X)\dots f_{r-1}^{r-1}(X).$$
and investigate the smoothness.
For any $j$ such that $(j,r)=1$ we define
$$F^{(j)} = \alpha^j\prod_{i=1}^{r-1} f_i^{ij \mod{r}}.$$
Then $F^{(1)}=F$, $F^{j} = F^{(j)}H^r$ for some $H\in\Ff_q[X]$ and $Y^r = F^{(j)}(X)$ is another affine model for the curve $C$.

If $F(x)\not=0$ then any of the models will be smooth at $x$. If $f_i(x)=0$ for some $x\in\Ff_q$ and for some $i$ such that $(i,r)=d$ then we can find some $j$ such that $(j,r)=1$ and $ij \equiv d \mod {r}$. Hence $F^{(j)}$ would have a $d^{th}$ root at $x$ and the model $Y^r = F^{(j)}(X)$ would be smooth at $x$ if and only if $d=1$.

Now, without loss of generality we may assume we have an affine model $Y^r=\alpha F(X)$ such that $f_d(x)=0$ for some $x\in\Ff_q$ and $d|r$, $d\not=1$. Moreover, without loss of generality, we may assume $x=0$. Blowing-up the curve at $(0,0)$ then we get the variety defined by
$$(Y^r - \alpha f_1^1(X)f_2^2(X)\dots f_{r-1}^{r-1}(X), Xw-Yz)$$
where $w,z$ are projective coordinates. If $z\not =0$, then $Y = Xw$ and by writing $f_d(X)=Xf'_d(X)$ we get
\begin{align*}
0=(Xw)^r - X^d\alpha f_1(X)f_2^2(X)\dots f'^d_d(X)\dots f_{r-1}^{r-1}(X)\\
= X^d \left((X^{\frac{r}{d}-1} w^{\frac{r}{d}})^d - \alpha f_1(X)f_2^2(X)\dots f'^d_d(X)\dots f_{r-1}^{r-1}(X) \right).
\end{align*}

%Moreover, if $f_i(x)=0$ for some $x\in\Pp^1(\Ff_q)$ and some $i$ such that $(i,r)=d\not=
%
%However, if $f_i$ has a root for $i \not = 1$, then this will not be smooth so we will need a different model.
%
%For any integers $i,n$, denote "$i \mod{n}$" to be the smallest non-negative integer congruent to $i$ mod $n$. Fix an $i$ such that $f_i$ has a root. Without loss of generality, we may assume that $f_i(0)=0$. Let $f_i(X)=Xf'_i(X)$ and $d=(r,i)$. Then we can find some $j$ such that $(j,r)=1$ and $ij \equiv d \mod{r}$. So consider the affine model for $C$
%$$Y^r = f_1(X)^{j \mod{r}}f_2(X)^{2j\mod{r}}\dots f_{r-1}(X)^{(r-1)j\mod{r}}$$
%Then in this model $f_i$ would appear with power $d$ so we may assume that $i=d|r$.
%
%Fix $d|r$ and assume $f_d(0)=0$. We need to find a model of this curve that will be smooth at $0$. Clearly if $d=1$, we are done. If $d\not=1$, then we will blow-up the curve at $(0,0)$. That is we will look at the variety defined by
%$$(Y^r - f_1^1(X)f_2^2(X)\dots f_{r-1}^{r-1}(X), Xw-Yz)$$
%where $w,z$ are projective coordinates. If $z\not =0$, then we get $Y = Xw$ and hence
%\begin{align*}
%0=(Xw)^r - X^df_1(X)f_2^2(X)\dots f'^d_d(X)\dots f_{r-1}^{r-1}(X)\\
%= X^d \left((X^{\frac{r}{d}-1} w^{\frac{r}{d}})^d - f_1(X)f_2^2(X)\dots f'^d_d(X)\dots f_{r-1}^{r-1}(X) \right)
%\end{align*}

If we let $Y' = X^{\frac{r}{d}-1}w^{\frac{r}{d}}$, we get the affine model
$$Y'^d = \alpha f_1(X)f_2^2(X)\dots f'^d_d(X)\dots f_{r-1}^{r-1}(X)$$
which is birationally equivalent to
$$Y^d = \alpha \prod_{i=1}^{d-1}\left(\prod_{j=0}^{\frac{r}{d}-1} f_{jd+i}(X)\right)^{i} := F_{(d)}(X).$$
Now $f_d(X) \not | F_{(d)}(X)$, hence $F_{(d)}(0)\not=0$ and the affine model will be smooth at $0$. Further note that $F_{(r)}(x)=F(x)$.

%This motivates defining for each $1 \leq k \leq r$, let $d_k=(k,r)$ and
%$$F_k(X) := \prod_{i=1}^{r-1} f_i^{\frac{ir}{d_k} \mod{d_k}}(X)$$
%%(Note that we can find polynomials $H_k$ such that $F_1^k = F_k^{d_k}H_k^r$)
%Notice that if $d|r$ then
%$$F_d(X) = \prod_{i=1}^{d-1}\left(\prod_{j=0}^{\frac{r}{d}-1} f_{jd+i}(X)\right)^{i}$$
%
%Therefore, for any $x\in\mathbb{F}_q$, if $f_j(x)\not =0$ for all $j$, then $x$ is smooth in the model
%$$Y^r = F_r(X) = \alpha f_1(X)f_2^2(X)\dots f_{r-1}^{r-1}(X)$$
%If, $f_i(x)=0$ for some $(i,r)=1$, then find a $j$ such that $ij \equiv 1 \mod{r}$ and $x$ is smooth in the model
%$$Y^r = F^{(j)}(X)$$
%If, however, $f_j(x)=0$ for $(j,r)=d\not=1$, then $x$ is smooth in the model
%$$Y^d = F_{d}(X)$$
%as $F_{d}$ now has no root at $x$.

Further, if $x_{q+1}$ is the point at infinity then
$$F_{(d)}(x_{q+1}) = \begin{cases} \alpha & F\in\F^{jd}_{(d_1,\dots,d_{r-1})}, 1\leq j \leq r/d-1 \\ 0 & \mbox{otherwise}  \end{cases}.$$

This leads to the following lemma

\begin{lem}\label{numpts1}
Let $C$ be the smooth projective curve associated to the field $L=\Ff_q(X)\left(\sqrt[r]{f_1(X)f_2^2(X)\dots f_{r-1}^{r-1}(X)}\right)$ then
$$\#C(\Pp^1(\Ff_q)) = q+1 + \sum_{d|r}\sum_{\substack{i=1\\(i,d)=1}}^{d-1}  \sum_{x\in\mathbb{P}^1(\Ff_q)} \chi_{d}^i(F_{(d)}(x)).$$
\end{lem}

\begin{proof}

If $x$ is not a root of any of the $f_i$, then there will be $r$ points lying over $x$, if $F_{(r)}(x)$ is an $r^{th}$ power and no points otherwise. We can write this as
$$1+\sum_{i=1}^{r-1}\chi_r(F^i_{(r)}(x)) = 1+ \sum_{d|r}\sum_{\substack{i=1\\(i,d)=1}}^{d-1} \chi_d^i(F_{(d)}(x)).$$
If $f_i(x)=0$ for some $(i,r)=1$, then there will be one point lying over $x$. Further in this case $F_d(x)=0$ for all $d|r$ so we can write this as
$$1+ \sum_{d|r}\sum_{\substack{i=1\\(i,d)=1}}^{d-1} \chi_d^i(F_{(d)}(x)).$$
If $f_i(x)=0$ for some $(i,r)=d\not=1$, then we have to look at the smooth model $y^d=F_{(d)}(x)$. Thus there will be $d$ points lying over $x$ if $F_{(d)}(x)$ is a $d^{th}$ power and no points otherwise. We can write this as
$$1+\sum_{i=1}^{d}\chi_d(F^i_{(d)}(x)) = 1+ \sum_{d'|d}\sum_{\substack{i=1\\(i,d')=1}}^{d'-1} \chi_{d'}^i(F_{(d')}(x)).$$
Further for any $d'| r$ such that $d' \not | d$ we get that the exponent of $f_i$ in $F_{(d')}$ is non-zero. Hence $F_{(d')}(x)=0$. Therefore, regardless of the behavior at $x$, the number or points lying above $x$ is
$$1+ \sum_{d|r}\sum_{\substack{i=1\\(i,d)=1}}^{d-1} \chi_d^i(F_{(d)}(x)).$$
Summing up over all $x$, we find that
\begin{align*}
\#C(\Ff_q) &= \sum_{x\in\mathbb{P}^1(\Ff_q)}\left(1+ \sum_{d|r}\sum_{\substack{i=1\\(i,d)=1}}^{d-1} \chi_{d}^i(F_{(d)}(x))\right)\\
&=q+1 + \sum_{d|r}\sum_{\substack{i=1\\(i,d)=1}}^{d-1}  \sum_{x\in\mathbb{P}^1(\Ff_q)} \chi_{d}^i(F_{(d)}(x)).
\end{align*}
\end{proof}

\section{Set Count}\label{setcountsec}

First we need two lemmas from other papers. The first one is Lemma 6.4 from \cite{LMM} saying

\begin{lem}[Lemma 6.4 from \cite{LMM}]\label{setcount1}

Let $d_1,\dots,d_n$ be positive integers. For $0\leq\ell\leq q$ let $x_1,\dots,x_\ell$ be distinct elements of $\Ff_q$. Let $U\in\Ff_q[x]$ be such that $U(x_i)\not=0$ for $1\le i \leq \ell$. Let $a_{1,1},\dots,a_{\ell,1},\dots,a_{1,n},\dots,a_{\ell,n} \in \Ff_q^*$. Then the size of
\begin{align*}
\{(f_1,\dots,f_n) \in \F_{d_1}\times\dots\times\F_{d_n} : (f_i,f_j)=(f_i,U)=1, f_j(x_i) = a_{i,j}, 1\leq i \leq \ell, 1 \leq j \leq n\}
\end{align*}
is
\begin{align*}
R_n^U(\ell) := \frac{L_{n-1}q^{d_1+\dots+d_n}}{\zeta_q^n(2)} \left(\frac{q}{(q-1)^n(q+n)}\right)^\ell \prod_{P|U} \frac{|P|}{|P|+n}\left(1+O\left(q^{-\frac{\min{d_i}}{2}}\right)\right)
\end{align*}
where
\begin{align*}
L_n = \prod_{j=1}^n K_j
\end{align*}
and
\begin{align*}
K_j = \prod_P \left(1-\frac{j}{(|P|+1)(|P|+j)}\right).
\end{align*}
\end{lem}

From now on, if $K_j$ or $L_n$ appears elsewhere in the paper, it will be the same formula that appears in the above Lemma \ref{setcount1}. The next lemma is Lemma 3.2 from \cite{BDFL1}

\begin{lem}[Lemma 3.2 from \cite{BDFL1}]\label{multfunc}

Fix $x_1,\dots,x_\ell\in\Ff_q$ and let $U$ be a polynomial of degree $u$ with $U(x_i)\not=0$, $1\leq i\leq\ell$. Define the multiplicative function

\begin{align*}
c_j^U(F) = \begin{cases} \mu^2(F) \prod_{P|F} \left(1+j|P|^{-1}\right)^{-1} & F(x_i)\not=0, 1\leq i\leq\ell, (F,U)=1 \\ 0 & \mbox{otherwise} \end{cases}
\end{align*}

Then for any $0<\epsilon< 1$

\begin{align*}
\sum_{\deg(F)=d} c_j^U(F) = \frac{K_jq^d}{\zeta_q(2)} \left(\frac{q+j}{q+j+1}\right)^{\ell} \left(\prod_{P|U}\left(\frac{|P|+j}{|P|+j+1}\right)\right)\left(1+O\left(q^{\epsilon(d+u+\ell)-d}\right)\right).
\end{align*}
\end{lem}

We use these two lemmas to prove the following lemma.

\begin{lem}\label{setcount2}

Let $d_{1,1},\dots,d_{r_1-1,1},\dots,d_{1,n},\dots,d_{r_n-1,n}$ be positive integers. For $0\leq\ell\leq q$ let $x_1,\dots,x_\ell$ be distinct elements of $\Ff_q$. Let $U\in\Ff_q[x]$ be such that $U(x_i)\not=0$ for $1\le i \leq \ell$. Let $a_{1,1},\dots,a_{\ell,1},\dots,a_{1,n},\dots,a_{\ell,n} \in \Ff_q^*$. Then the size of
\begin{align*}
\{(F_1,\dots,F_n) \in \F_{d_{1,1},\dots,d_{r_1-1,1}}\times\dots\times\F_{d_{1,n},\dots,d_{r_n-1,n}} : & (F_j,F_k)=(F_j,U)=1, F_j(x_i) = a_{i,j},\\ & 1\leq i \leq \ell, 1 \leq j,k \leq n, j\not=k\}
\end{align*}
is
\begin{align*}
T_n^U(\ell) :=& \frac{L_{r_1+\dots+r_n-n-1}q^{d_{1,1}+\dots+d_{r_n-1,n}}}{\zeta_q(2)^{r_1+\dots+r_n-n}} \left(\frac{q}{(q-1)^n(q+r_1+\dots+r_n-n)}\right)^\ell\\
& \times  \prod_{P|U} \frac{|P|}{|P|+r_1+\dots+r_n-n}\left(1+O\left(q^{-\frac{\min{d_{i,j}}}{2}}\right)\right).
\end{align*}
\end{lem}

\begin{proof}
If we write $F_i = \prod_{j=1}^{r_i-1} f_{i,j}^j$ then,
\begin{align*}
T^U_n(\ell) &= \sum_{\substack{f_{i,j} \in \F_{d_{i,j}}, j\not=1 \\(f_{i,j},f_{k,h})=1 \\ f_{i,j}(x_m)\not=0}} R_n^{U\prod_{j\not=1}f_{i,j}}(\ell)\\
& = \sum_{\substack{f_{i,j} \in \F_{d_{i,j}}, j\not=1 \\(f_{i,j},f_{k,h})=1 \\ f_{i,j}(x_m)\not=0}} \frac{L_{n-1}q^{d_{1,1}+\dots+d_{1,n}}}{\zeta_q^n(2)} \left(\frac{q}{(q-1)^n(q+n)}\right)^\ell \prod_{P|U\prod_{j\not=1} f_{i,j}} \frac{|P|}{|P|+n}\left(1+O\left(q^{-\frac{\min{d_{1,j}}}{2}}\right)\right) \\
& = \frac{L_{n-1}q^{d_{1,1}+\dots+d_{1,n}}}{\zeta_q^n(2)} \left(\frac{q}{(q-1)^n(q+n)}\right)^\ell \left(1+O\left(q^{-\frac{\min{d_{1,j}}}{2}}\right)\right) \sum_{\substack{f_{i,j} \in \F_{d_{i,j}}, j\not=1 \\(f_{i,j},f_{k,h})=1 \\ f_{i,j}(x_m)\not=0}} \prod_{P|U\prod_{j\not=1} f_{i,j}} \frac{|P|}{|P|+n}.
\end{align*}

Lemma \ref{multfunc} shows that

\begin{align*}
\sum_{\substack{f_{i,j} \in \F_{d_{i,j}}, j\not=1 \\(f_{i,j},f_{k,h})=1 \\ f_{i,j}(x_m)\not=0}} \prod_{P|U\prod_{j\not=1} f_{i,j}} \frac{|P|}{|P|+n} = \sum_{\deg(f_{2,1}) = d_{2,1}}c^U_{n}(f_{2,1})\sum_{\deg(f_{2,2}) = d_{2,2}}c_n^{Uf_{2,1}}(f_{2,2})\dots \\
\sum_{\deg(f_{r_n-1,n}) = d_{r_n-1,n}}c_n^{U\prod_{i\not=1}f_{i,j}}(f_{r_n-1,n})
\end{align*}
\begin{align*}
=& \frac{L_{r_1+\dots+r_n-n-1}q^{d_{1,1}+\dots+d_{r_n-1,n}}}{\zeta_q(2)^{r_1+\dots+r_n-n}} \left(\frac{q}{(q-1)^n(q+r_1+\dots+r_n-n)}\right)^\ell\\
& \times  \prod_{P|U} \frac{|P|}{|P|+r_1+\dots+r_n-n}\left(1+O\left(q^{-\epsilon\min{d_{i,j}}}\right)\right).
\end{align*}

Thus, for an appropriate choice of $\epsilon$, we get the result.

\end{proof}

Lemma \ref{setcount2} deals with the case where the $F_i$ are all coprime. We want, however the case where they are not necessarily coprime. Suppose now that we have $(F_1,\dots,F_n) \in\F_{d_{1,1},\dots,d_{r_1-1,1}}\times\dots\times\F_{d_{1,n},\dots,d_{r_n-1,n}}$ which are not necessarily coprime. Let
$$F_j = \prod_{k=1}^{r_j-1} f_{j,k}^k.$$
and we want to rewrite the $F_j$ as products of square-free polynomials that are all coprime to one another. For example, if $n=2,r_1=r_2=4$ then we would have
$$F_1= f_{1,1}f_{1,2}^2f_{1,3}^3 \quad \quad \quad F_2 = f_{2,1}f_{2,2}^2f_{2,3}^3.$$
Then if we define
$$f_{(i,j)} = \gcd(f_{1,i},f_{2,j}), 1 \leq i,j\leq 3$$
and
$$f_{(i,0)} = \frac{f_{1,i}}{f_{(i,1)}f_{(i,2)}f_{(i,3)}} \quad f_{(0,j)} = \frac{f_{2,j}}{f_{(1,j)}f_{(2,j)}f_{(3,j)}}, 1\leq i,j \leq 3.$$

%$$f_{(1,1)} = \gcd(f_{1,1},f_{2,1}) \quad f_{(1,2)} = \gcd(f_{1,1},f_{2,2}) \quad f_{(1,3)} = \gcd(f_{1,1},f_{2,3})$$
%$$f_{(2,1)} = \gcd(f_{1,2},f_{2,1}) \quad f_{(2,2)} = \gcd(f_{1,2},f_{2,2}) \quad f_{(2,3)} = \gcd(f_{1,2},f_{2,3})$$
%$$ f_{(3,1)} = \gcd(f_{1,3},f_{2,1}) \quad f_{(3,2)} = \gcd(f_{1,1},f_{3,2}) \quad f_{(3,3)} = \gcd(f_{1,3},f_{2,3})$$
%$$f_{(1,0)} = \frac{f_{1,1}}{f_{(1,1)}f_{(1,2)}f_{(1,3)}} \quad f_{(2,0)} = \frac{f_{1,2}}{f_{(2,1)}f_{(2,2)}f_{(2,3)}} \quad f_{(3,0)} = \frac{f_{1,3}}{f_{(3,1)}f_{(3,2)}f_{(3,3)}}$$
%$$ f_{(0,1)} = \frac{f_{2,1}}{f_{(1,1)}f_{(2,1)}f_{(3,1)}} \quad f_{(0,2)} = \frac{f_{2,2}}{f_{(1,2)}f_{(2,2)}f_{(3,2)}} \quad f_{(0,3)} = \frac{f_{2,3}}{f_{(1,3)}f_{(2,3)}f_{(3,3)}} \quad $$

Then we can all the $f_{(i,j)}$ are square-free and coprime to one another. Moreover
$$F_1 = f_{(1,0)}f_{(1,1)}f_{(1,2)}f_{(1,3)}f_{(2,0)}^2f_{(2,1)}^2f_{(2,2)}^2f_{(2,3)}^2f_{(3,0)}^3f_{(3,1)}^3f_{(3,2)}^3 f_{(3,3)}^3 = \underset{(i,j)\not=(0,0)}{\prod_{i=0}^3 \prod_{j=0}^3} f_{(i,j)}^i$$
$$F_2 = f_{(0,1)}f_{(1,1)}f_{(2,1)}f_{(3,1)}f_{(0,2)}^2f_{(1,2)}^2f_{(2,2)}^2f_{(3,2)}^2f_{(0,3)}^3f_{(1,3)}^3f_{(2,3)}^3 f_{(3,3)}^3 =  \underset{(i,j)\not=(0,0)}{\prod_{i=0}^3 \prod_{j=0}^3} f_{(i,j)}^j.$$

In general, define
\begin{align*}
\R = [0,\dots,r_1-1]\times\dots\times[0,\dots,r_n-1]\setminus\{(0,0,\dots,0)\}
\end{align*}
and write $\vec{\alpha}\in\R$ as $\vec{\alpha}=(\alpha_1,\dots,\alpha_n)$. Define also $f_{\vec{\alpha}}$ to be the largest polynomial such that
$$f_{\vec{\alpha}} \mbox{ divides } \gcd_{\substack{j=1,\dots,n \\ \alpha_j \not=0}}(f_{j,\alpha_j}) \quad  \mbox{ and } \quad (f_{\vec{\alpha}}, \prod_{\substack{j=1 \\ \alpha_j=0}}^n F_j)=1.$$

With this definition we get if $\vec{\alpha}\not=\vec{\beta}$ then $(f_{\vec{\alpha}},f_{\vec{\beta}})=1$. Indeed, suppose we have $\alpha_j\not=\beta_j$ and $\alpha_j,\beta_j\not=0$. Then $f_{\vec{\alpha}}|f_{j,\alpha_j}$ and $f_{\vec{\beta}}|f_{j,\beta_j}$ and since $(f_{j,\alpha_j},f_{j,\beta_j})=1$, we get that $(f_{\vec{\alpha}},f_{\vec{\beta}})=1$. On the other hand suppose we have $\alpha_j\not=\beta_j=0$. Then $f_{\vec{\alpha}}|f_{j,\alpha_j}|F_j$ but $(f_{\vec{\beta}},F_j)=1$ hence $(f_{\vec{\alpha}},f_{\vec{\beta}})=1$.

Now, define $\vec{\beta} = (0,\dots,0,k,0,\dots,0)$, where the $k$ is in the $j^{th}$ position. Then $f_{\vec{\beta}}$ is the largest polynomial that divides $f_{j,k}$ that is coprime to all the other $f_{j',k'}$. If $f_{\vec{\beta}}\not = f_{j,k}$ then
$$\prod_{\substack{\vec{\alpha} \not = \vec{\beta} \\ \alpha_j =k}} f_{\vec{\alpha}}$$
is the largest polynomial that divides $f_{j,k}$ that is not coprime to at least one of the other $f_{j',k'}$. If $f_{\vec{\beta}} = f_{j,k}$, then $f_{\vec{\alpha}} =1$ for all $\vec{\alpha}\not=\vec{\beta}$ such that $\alpha_j=k$. In either case we get
$$f_{j,k} = f_{\vec{\beta}} \prod_{\substack{\vec{\alpha}\not=\vec{\beta} \\ \alpha_j=k}}f_{\vec{\alpha}} = \prod_{\alpha_j=k}f_{\vec{\alpha}}.$$
We then rewrite
$$F_j = \prod_{k=1}^{r_j-1}f_{j,k}^k \quad \quad \mbox{ as } \quad \quad F_j = \prod_{\vec{\alpha}\in\R} f_{\vec{\alpha}}^{\alpha_j}.$$

Define $\vec{d}(\vec{\alpha}) := (d(\vec{\alpha}))_{\vec{\alpha}\in\R}$ to be an integer vector with non-negative entries indexed by the vectors of $\R$. Further, define the set
$$\F_{\vec{d}(\vec{\alpha})} = \{(f_{\vec{\alpha}})_{\vec{\alpha}\in\R}\in\prod_{\vec{\alpha}\in\R}\F_{d(\vec{\alpha})} : (f_{\vec{\alpha}},f_{\vec{\beta}})=1 \mbox{ for all } \vec{\alpha}\not=\vec{\beta}\in\R\}.$$

To ease notation, we will write just $(f_{\vec{\alpha}})$ instead of $(f_{\vec{\alpha}})_{\alpha\in\R}$ if it is clear what set the indices $\vec{\alpha}$ run over.
Hence,
\begin{align*}
&\{(F_1,\dots,F_n) \in \F_{d_{1,1},\dots,d_{r_1-1,1}}\times\dots\times\F_{d_{1,n},\dots,d_{r_n-1,n}} : F_j(x_i) = a_{i,j}, 1\leq i \leq \ell, 1 \leq j \leq n\}\\
&= \bigcup_{\substack{\vec{d}(\vec{\alpha}) \\ \sum_{\alpha_j=k} d(\vec{\alpha}) = d_{j,k} }}\{(f_{\vec{\alpha}}) \in \F_{\vec{d}(\vec{\alpha})} :  \prod_{\vec{\alpha}\in\R}f_{\vec{\alpha}}^{\alpha_j}(x_i) = a_{i,j}, 1\leq i \leq \ell, 1 \leq j \leq n\}
\end{align*}

%\textbf{Some comments on the notation:} If we write $\F_{d_A}$, without the brackets on $d_A$, we will mean it to be the set of monic, square-free polynomials of degree $d_A$, corresponding with the notation in Section \ref{intro}. With the brackets on $d_A$, we will take $(d_A)$ to mean a tuple indexed by $A\in\R$ and therefore $\F_{(d_A)}$ will be the set of tuples of square-free, coprime, monic polynomials, indexed by $A\in\R$ with degrees $d_A$.

This leads to Proposition \ref{setcount}

\begin{prop}\label{setcount}

Let $\vec{d}(\vec{\alpha})$ be as above. For $0\leq\ell\leq q$ let $x_1,\dots,x_\ell$ be distinct elements of $\Ff_q$. Let $a_{1,1},\dots,a_{\ell,1},\dots,a_{1,n},\dots,a_{\ell,n} \in \Ff_q^*$. Then the size of
\begin{align*}
\{(f_{\vec{\alpha}}) \in \F_{\vec{d}(\vec{\alpha})} : F_j(x_i) := \prod_{\vec{\alpha}\in\R} f_{\vec{\alpha}}^{\alpha_j}(x_i) = a_{i,j}, 1\leq i \leq \ell, 1 \leq j \leq n\}
\end{align*}
is
\begin{align*}
S_n(\ell) := \frac{L_{r_1\cdots r_n-2} q^{\sum_{\vec{\alpha}\in\R}d_{\vec{\alpha}}}}{\zeta_q(2)^{r_1\cdots r_n-1}} \left(\frac{q}{(q-1)^n(q+r_1\cdots r_n-1)}\right)^\ell\left(1 + O\left(q^{-\frac{\min_{\vec{\alpha}\in\R} d_{\vec{\alpha}}}{2}}\right)\right).
\end{align*}
\end{prop}

\begin{proof}

We will split the $F_j:=\prod_{\vec{\alpha}\in\R}f_{\vec{\alpha}}^{\alpha_j}$ into their coprime parts. In order to do this we will need some new notation. Define
\begin{align*}
\s_j := \{(0,\dots,0,\alpha_j,0,\dots,0): 1\leq \alpha_j \leq r_j-1\} \subset \R
\end{align*}
where the non-zero entry is in the $j^{th}$ coordinate. Define,
\begin{align*}
\s = \bigcup_{j=1}^n \s_j.
\end{align*}
Then the factor of $F_j = \prod_{\vec{\alpha}\in\R}f_{\vec{\alpha}}^{\alpha_j}$ that is coprime to all $F_i$ such that $i\not=j$ will be $\prod_{\vec{\alpha}\in\s_j}f_{\vec{\alpha}}^{\alpha_j}$. Further for any subset $\R'\subset \R$, define
\begin{align*}
\F^{\R'}_{(d(\vec{\alpha}))_{\vec{\alpha}\in\R'}} = \{(f_{\vec{\alpha}}) \in \prod_{\vec{\alpha} \in \R'}\F_{d(\vec{\alpha})} : (f_{\vec{\alpha}},f_{\vec{\beta}}) = 1, \mbox{ for } \alpha\not=\beta \in \R' \}.
\end{align*}
We will denote this as just $\F^{\R'}_{\vec{d}(\vec{\alpha})}$ with the understanding that in this context $\vec{d}(\vec{\alpha})$ is indexed by $\R'$ instead of $\R$.
Then,

\begin{align*}
S_n(\ell) &= |\{(f_{\vec{\alpha}}) \in \F_{\vec{d}(\vec{\alpha})} : \prod_{\vec{\alpha}\in\R}f_{\vec{\alpha}}^{\alpha_j}(x_i) = a_{i,j}, 1\leq i \leq \ell, 1 \leq j \leq n\}|\\
&= \sum_{\substack{ (f_{\vec{\alpha}})\in\F^{\R\setminus\s}_{\vec{d}(\vec{\alpha})} \\ f_{\vec{\alpha}}(x_i)\not=0 }} |\{ (f_{\vec{\alpha}}) \in \F^\s_{\vec{d}(\vec{\alpha})} : (f_{\vec{\alpha}},f_{\vec{\beta}}) =1, \vec{\alpha}\in\s, \vec{\beta}\in\R\setminus\s  \prod_{\vec{\alpha}\in\s_j}f_{\vec{\alpha}}^{\alpha_j}(x_i) = b_{i,j}, 1\leq j \leq n, 1 \leq i \leq \ell \}|\\
& = \sum_{\substack{ (f_{\vec{\alpha}})\in\F^{\R\setminus\s}_{\vec{d}(\vec{\alpha})} \\ f_{\vec{\alpha}}(x_i)\not=0 }} T_n^{\prod_{\vec{\alpha}\in\R\setminus\s} f_{\vec{\alpha}}}(\ell)
\end{align*}
where $T_n^{\prod_{\vec{\alpha}\in\R\setminus\s} f_{\vec{\alpha}}}(\ell)$ is as in Lemma \ref{setcount2} and
$$b_{i,j} = a_{i,j} \prod_{\vec{\alpha}\in\R\setminus S} f_{\vec{\alpha}}(x_i)^{-\alpha_j}.$$

Thus, similarly to the previous lemma,

\begin{align*}
S_n(\ell) &= \sum_{\substack{ (f_{\vec{\alpha}})\in\F^{\R\setminus\s}_{\vec{d}(\vec{\alpha})} \\ f_{\vec{\alpha}}(x_i)\not=0}} \frac{L_{r_1+\dots+r_n-n-1}q^{\sum_{\vec{\alpha}\in\s}\vec{d}(\vec{\alpha})}} {\zeta_q(2)^{r_1+\dots+r_n-n}} \left(\frac{q}{(q-1)^n(q+r_1+\dots+r_n-n)}\right)^\ell\\
&\times  \prod_{\substack{P|f_{\vec{\alpha}} \\ \vec{\alpha}\in \R\setminus\s}} \frac{|P|}{|P|+r_1+\dots+r_n-n}\left(1+O\left(q^{-\frac{\min_{\vec{\alpha}\in\s}{d_{\vec{\alpha}}}}{2}}\right)\right)\\
& = M \frac{L_{r_1+\dots+r_n-n-1}q^{\sum_{\vec{\alpha}\in\R}d_{\vec{\alpha}}}}{\zeta_q(2)^{r_1+\dots+r_n-n}} \left(\frac{q}{(q-1)^n(q+r_1+\dots+r_n-n)}\right)^\ell \left(1+O\left(q^{-\frac{\min_{\vec{\alpha}\in\s}{d_{\vec{\alpha}}}}{2}}\right)\right).
\end{align*}

By the same reasoning as in the proof of Lemma \ref{setcount2} we get
\begin{align*}
M &= \sum_{\substack{ (f_{\vec{\alpha}})\in\F^{\R\setminus\s}_{\vec{d}(\vec{\alpha})} \\ f_{\vec{\alpha}}(x_i)\not=0}} \prod_{\substack{P|f_{\vec{\alpha}} \\ \vec{\alpha}\in \R\setminus\s}} \frac{|P|}{|P|+r_1+\dots+r_n-n}\\
&= \frac{L_{r_1\dots r_n-1}q^{\sum_{{\vec{\alpha}}\in\R\setminus\s} d_{\vec{\alpha}}}} {L_{r_1+\dots+r_n-n-1}\zeta_q(2)^{r_1\cdots r_n-r_1-\dots r_n +n-1}}\left(\frac{q+r_1+\dots+r_n-n-1}{q+r_1\cdots r_n-1}\right)^\ell\left(1 + O\left(q^{-\epsilon\min_{{\vec{\alpha}}\in\R\setminus\s}d_{\vec{\alpha}}}\right)\right).
\end{align*}

Therefore, if we choose an appropriate $\epsilon$, we get the result.

\end{proof}

\section{The case $r=p^n$}\label{primepower}

In order to apply the counting formula of Proposition \ref{setcount} we want to write the $n$ functions defined by
$$F_{(p^j)}(X) = \prod_{i=1}^{p^n-1} f_i^{i \mod{p^j}} \quad 1 \leq j \leq n$$
in terms of coprime functions.

To apply the results of Section \ref{setcountsec} we define, in this context

%We prove the result in the section for the case where $r$ is a power of a prime as it uses all the basic features of the general case but with less notation required. Previously we have said our curve would have an affine model of the form
%$$Y^{p^n}=F(X):=\alpha f_1(X)f_2(X)^2\dots f_{p^n-1}(X)^{p^n-1}.$$
%For easier notation we will redefine
$$\R = [0,\dots,p-1]^n\setminus\{(0,\dots,0)\}.$$
If $\vec{\alpha}\in\R$ we will denote it $\vec{\alpha}=(\alpha_1,\dots,\alpha_n)$. Rewrite
$$F := \prod_{\vec{\alpha}\in\R} f_{\vec{\alpha}}^{\alpha_1+p\alpha_2+\dots+p^{n-1}\alpha_n}.$$
That is we make the identification $f_{\alpha_1+p\alpha_2+\dots+p^{n-1}\alpha_n}\to f_{(\alpha_1,\dots,\alpha_n)}$

Recall the notation in Section \ref{Numpts} where we denoted, for $d|r$,
$$F_{(d)}(X) :=  \prod_{i=1}^{d-1}\left(\prod_{j=0}^{\frac{r}{d}-1} f_{jd+i}(X)\right)^{i} = \prod_{i=1}^{r-1} f_i(X)^{i \mod{d}}.$$
and we have
$$F_{(p^j)}(X) =  \prod_{\vec{\alpha}\in\R}f_{\vec{\alpha}}(X)^{\alpha_1+p\alpha_2+\dots+p^{j-1}\alpha_j} $$
for $j=1,\dots,n$. Further, if necessary, we denote $F_{(1)}(X)=1$.

Define $d_1:=d(1,0,\dots,0)$. Also use the notation $f_1 = f_{(1,0,\dots,0)}$. We will redefine the sets in Section \ref{intro} using this new notation.
$$\R' = \R\setminus \{(1,0,\dots,0)\}$$
$$\hat{\F}_{\vec{d}(\vec{\alpha})} = \{(f_1, (f_{\vec{\alpha}})) \in \hat{\F}_{d_1} \times \prod_{\vec{\alpha} \in \R'}\F_{d(\vec{\alpha})} : (f_{\vec{\alpha}},f_{\vec{\beta}}) = 1 \mbox{ for all } \vec{\alpha}\not =\vec{\beta}\}$$

Let
$$\delta(\vec{\alpha},\vec{\beta}) = \begin{cases} 1 & \vec{\alpha} = \vec{\beta} \\ 0 & \vec{\alpha} \not = \vec{\beta} \end{cases}.$$
Then define
$$\F^{\vec{\beta}}_{\vec{d}(\vec{\alpha})} = \{(f_{\vec{\beta}},(f_{\vec{\alpha}})) \in \F_{d(\vec{\beta})-1}\times \prod_{\vec{\alpha}\not=\vec{\beta}} \F_{d(\vec{\alpha})} : (f_{\vec{\alpha}},f_{\vec{\gamma}})=1, \vec{\alpha}\not=\vec{\gamma}\in\R \}.$$
Likewise define $\hat{\F}^{\vec{\beta}}_{\vec{d}(\vec{\alpha})}$ to be elements of $\F^{\vec{\beta}}_{\vec{d}(\vec{\alpha})}$ where $f_1$ is not necessarily monic. Further, define
$$\F_{[\vec{d}(\vec{\alpha})]} = \F_{\vec{d}(\vec{\alpha})}\cup\bigcup_{\vec{\beta}\in\R} {\F}^{\vec{\beta}}_{\vec{d} (\vec{\alpha})}$$
$$\hat{\F}_{[\vec{d}(\vec{\alpha})]} = \hat{\F}_{\vec{d}(\vec{\alpha})}\cup\bigcup_{\vec{\beta}\in\R} \hat{\F}^{\vec{\beta}}_{\vec{d}({\vec{\alpha}})}.$$

Now, if $x_{q+1}$ is the point at infinity, then,

$$F_{(p^j)}(x_{q+1}) = \begin{cases} \mbox{leading coefficient of } f_1 & (f_{\vec{\alpha}})\in\hat{\F}^{\beta}_{\vec{d}(\vec{\alpha})}\\ 0 & \mbox{otherwise}  \end{cases}$$
where $\vec{\beta}=(\beta_1,\dots,\beta_n)\in\R$ is any tuple such that $\beta_i=0$ for all $i\leq j$.

\begin{prop}\label{primepowerprop}

Let $a_{i,j} \in \Ff^*_q$, $1\leq j \leq n, 1 \leq i \leq \ell$ such that $a_{i,j} = a_{i,j-1}(b_{i,j})^{p^{j-1}}$ for some $b_{i,j} \in \Ff^*_q$. We let $a_{i,0}=1$ so that $a_{i,1}=b_{i,1}$. Then
\begin{align*}
&|\{(f_{\vec{\alpha}})\in\F_{\vec{d}(\vec{\alpha}))} : F_{(p^j)}(x_i) = a_{i,j}, 1\leq j\leq n, 1 \leq i \leq \ell\}| \\
&= \frac{L_{p^n-2} q^{\sum_{\vec{\alpha}\in\R}\vec{d}(\vec{\alpha})}}{\zeta_q(2)^{p^n-1}} \left(\frac{p^{\frac{n(n-1)}{2}}q} {(q-1)^n(q+p^n-1)}\right)^\ell\left(1 + O\left(q^{-\frac{\min_{\vec{\alpha}\in\R}d_{\vec{\alpha}}}{2}}\right)\right).
\end{align*}

\end{prop}

\begin{proof}

For $j=1,\dots,n$ define $F_{j} = \prod_{\vec{\alpha}\in\R}f_{\vec{\alpha}}^{\alpha_j}$. Then we can write $F_{(p^j)} = F_{(p^{j-1})}F_{j}^{p^{j-1}}$, where, again, we have $F_1=1$. Then $F_{(p^j)}(x_i)=a_{i,j}$ for all $i,j$ is equivalent to $F_{j}(x_i)=\epsilon_{i,j}b_{i,j}$ for all $i,j$ for some $\epsilon_{i,j}\in\mu_{p^{j-1}}$. Hence,

\begin{align*}
&|\{(f_{\vec{\alpha}})\in\F_{\vec{d}(\vec{\alpha})} : F_{(p^j)}(x_i) = a_{i,j}, 1\leq j\leq n, 1 \leq i \leq \ell\}| \\
& = |\{(f_{\vec{\alpha}})\in\F_{\vec{d}(\vec{\alpha})} : F_{j}(x_i) = \epsilon_{i,j}b_{i,j}, \epsilon_{i,j}\in \mu_{p^{j-1}},1\leq j\leq n, 1 \leq i \leq \ell\}|\\
& = \frac{L_{p^n-2} q^{\sum_{\vec{\alpha}\in\R}d(\vec{\alpha})}}{\zeta_q(2)^{p^n-1}} \left(\frac{p^{\frac{n(n-1)}{2}}q} {(q-1)^n(q+p^n-1)}\right)^\ell\left(1 + O\left(q^{-\frac{\min_{\vec{\alpha}\in\R}d(\vec{\alpha})}{2}}\right)\right).
\end{align*}

Where the second equality comes from Proposition \ref{setcount} and that there are $(p^{\frac{n(n-1)}{2}})^\ell$ different choices for the $\epsilon_{i,j}$.

\end{proof}

Now let us look at what happens if some of the $a_{i,j}=0$. Note that if $a_{i,j}=0$ then $a_{i,j+1}=0$ as $F_{p^j}|F_{p^{j+1}}$. Thus we have

\begin{cor}\label{primepowercor1}

Let $a_{i,j} \in \Ff_q$, $1\leq j \leq n, 1 \leq i \leq \ell$ such that
$$a_{i,1},\dots,a_{i,k} \not=0 \mbox{ and } a_{i,k+1},\dots,a_{i,n}=0 \mbox{ for } \sum_{i=0}^{k-1}m_i+1 \leq i \leq \sum_{i=0}^{k-1}m_i+m_k$$
for $k=0,\dots,n-1$ and $a_{i,j} \not=0$ for all $j$ for $\sum_{i=0}^{n-1}m_i +1\leq i \leq \ell$. Further, if $a_{i,j}\not=0$ then $a_{i,j} = a_{i,j-1}(b_{i,j})^{p^{j-1}}$ for some $b_{i,j} \in \Ff^*_q$. Then
$$|\{(f_{\vec{\alpha}})\in\F_{\vec{d}(\vec{\alpha})} : F_{p^j}(x_i) = a_{i,j}, 1\leq j\leq n, 1 \leq i \leq \ell\}|$$
\begin{align*}
= &\frac{L_{p^n-2} q^{\sum_{\vec{\alpha}\in\R}d(\vec{\alpha})}}{\zeta_q(2)^{p^n-1}} \prod_{k=0}^{n-1}\left( \frac{(p-1)p^{n-k-1}(q-1)^{n-k}}{qp^{\frac{(n-k)(n+k-1)}{2}}} \right)^{m_k}\\
& \times \left(\frac{p^{\frac{n(n-1)}{2}}q} {(q-1)^n(q+p^n-1)}\right)^\ell\left(1 + O\left(q^{-\frac{\min_{\vec{\alpha}\in\R}d(\vec{\alpha})}{2}}\right)\right).
\end{align*}

\end{cor}

\begin{proof}

Consider $i$ such that $\sum_{i=0}^{k-1}m_i+1\leq i\leq\sum_{i=0}^{k-1}m_i+m_k$. Then $f_{\vec{\alpha}}(x_i)=0$ for some $\vec{\alpha}$ such that $\alpha_{k+1}\not=0$ but $\alpha_j=0$ for all $j<k+1$. There are $(p-1)p^{n-j-1}$ different such $\vec{\alpha}$. Fix a partition $m_k := \sum_{\vec{\alpha}} m_{k,\vec{\alpha}}$ where the sum is over all such $\vec{\alpha}$ defined above and $m_{k,\vec{\alpha}}$ is the number of times $f_{\vec{\alpha}}(x_i)=0$.

Define $f'_{\vec{\alpha}}$ as $f_{\vec{\alpha}}$ divided by its roots and $F'_{p^j}$ as the corresponding product of the $f'_{\vec{\alpha}}$. Now $F'_{p^j}(x_i)$ is determined by $F_{p^j}(x_i)$ for $j\leq k$. $F'_{p^j}(x_i)$ will be fixed, up to a $p^{j-1}$th root of unity, for $j+1\leq k\leq n$. Thus we get a factor of
$$\left(\prod_{j=k+1}^{n} \frac{q-1}{p^{j-1}}\right)^{m_k} = \left( \frac{(q-1)^{n-k}}{p^{\frac{(n-k)(n+k-1)} {2}}} \right)^{m_k}.$$
Summing up over all the partitions of $m_k$ we get
\begin{align*}
&|\{(f_{\vec{\alpha}})\in\F_{\vec{d}(\vec{\alpha})} : F_{p^j}(x_i) = a_{i,j}, 0\leq j\leq n, 1 \leq i \leq \ell\}|\\
= & \prod_{k=0}^{n-1}\left((p-1)p^{n-k-1}\right)^{m_k} \frac{L_{p^n-2} q^{\sum_{\vec{\alpha}\in\R}d(\vec{\alpha}) - \sum_{k=0}^{n-1} m_k}}{\zeta_q(2)^{p^n-1}} \prod_{k=0}^{n-1}\left( \frac{(q-1)^{n-j}}{p^{\frac{(n-k)(n+k-1)}{2}}} \right)^{m_k}\\
& \times \left(\frac{p^{\frac{n(n-1)}{2}}q} {(q-1)^n(q+p^n-1)}\right)^\ell \left(1 + O\left(q^{-\frac{\min_{\vec{\alpha}\in\R}d(\vec{\alpha})}{2}}\right)\right) \\
= & \frac{L_{p^n-2} q^{\sum_{\vec{\alpha}\in\R}d(\vec{\alpha})}}{\zeta_q(2)^{p^n-1}} \prod_{k=0}^{n-1}\left( \frac{(p-1)p^{n-k-1}(q-1)^{n-k}}{qp^{\frac{(n-k)(n+k-1)}{2}}} \right)^{m_k}\left(\frac{p^{\frac{n(n-1)}{2}}q} {(q-1)^n(q+p^n-1)}\right)^\ell \left(1 + O\left(q^{-\frac{\min_{\vec{\alpha}\in\R}d(\vec{\alpha})}{2}}\right)\right).
\end{align*}

\end{proof}

\begin{cor}\label{primepowercor2}

Let $\epsilon_{i,j} \in \mu_{p^j}\cup\{0\}$, $1\leq j \leq n, 1 \leq i \leq \ell$ such that
$$\epsilon_{i,1},\dots,\epsilon_{i,k} \not=0 \mbox{ and } \epsilon_{i,k+1},\dots,\epsilon_i^n=0 \mbox{ for } \sum_{i=0}^{k-1}m_i+1 \leq i \leq \sum_{i=0}^{k-1}m_i+m_k$$
for $k=0,\dots,n-1$ and $\epsilon_{i,j} \not=0$ for all $j$ for $\sum_{i=0}^{n-1}m_i +1\leq i \leq \ell$. Further, if $\epsilon_{i,j}\not=0$ then $\epsilon_{i,j-1} = (\epsilon_{i,j})^p$. Then
$$|\{(f_{\vec{\alpha}})\in\F_{\vec{d}(\vec{\alpha})} : \chi_{p^j}(F_{p^j}(x_i)) = \epsilon_{i,j}, 1\leq j\leq n, 1 \leq i \leq \ell\}|$$
$$ =  \frac{L_{p^n-2} q^{\sum_{\vec{\alpha}\in\R}d(\vec{\alpha})}}{\zeta_q(2)^{p^n-1}} \prod_{k=0}^{n-1}\left( \frac{(p-1)p^{2n-2k-1}}{q} \right)^{m_k}\left(\frac{q} {p^n(q+p^n-1)}\right)^\ell\left(1 + O\left(q^{-\frac{\min_{\vec{\alpha}\in\R}d(\vec{\alpha})}{2}}\right)\right).$$

\end{cor}

\begin{proof}
For $k=1,\dots,n$, $F_{p^k}(x_i)=0$ for $1\leq i \leq \sum_{j=0}^{k-1}m_j$, so for $\sum_{j=0}^{k-1}m_j+1 \leq i \leq \ell$, $F_{p^k}(x_i)$ has $\frac{q-1}{p^k}$ choices. Hence
\begin{align*}
&|\{(f_{\vec{\alpha}})\in\F_{\vec{d}(\vec{\alpha})} : \chi_{p^j}(F_{p^j}(x_i)) = \epsilon_{i,j}, 1\leq j\leq n, 1 \leq i \leq \ell\}|\\
= & \prod_{k=1}^n \left(\frac{q-1}{p^k}\right)^{\ell -\sum_{j=0}^{k-1}m_j }  \frac{L_{p^n-2} q^{\sum_{\vec{\alpha}\in\R}d(\vec{\alpha})}}{\zeta_q(2)^{p^n-1}} \prod_{k=0}^{n-1}\left( \frac{(p-1)p^{n-k-1}(q-1)^{n-k}}{qp^{\frac{(n-k)(n+k-1)}{2}}} \right)^{m_k}\\
& \times \left(\frac{p^{\frac{n(n-1)}{2}}q} {(q-1)^n(q+p^n-1)}\right)^\ell\left(1 + O\left(q^{-\frac{\min_{\vec{\alpha}\in\R}d(\vec{\alpha})}{2}}\right)\right) \\
= & \frac{L_{p^n-2} q^{\sum_{\vec{\alpha}\in\R}d(\vec{\alpha})}}{\zeta_q(2)^{p^n-1}} \prod_{k=0}^{n-1}\left( \frac{(p-1)p^{2n-2k-1}}{q} \right)^{m_k}\left(\frac{q} {p^n(q+p^n-1)}\right)^\ell\left(1 + O\left(q^{-\frac{\min_{\vec{\alpha}\in\R}d(\vec{\alpha})}{2}}\right)\right).
\end{align*}

\end{proof}

Up until now, we have been looking only at points in $\Ff_q$. What we need to look at however, is points in $\mathbb{P}^1(\Ff_q)$. This is taken care of my the following corollary

\begin{cor}\label{primepowercor3}

Let $\epsilon_{i,j} \in \mu_{p^j}\cup\{0\}$, $1\leq j \leq n, 1 \leq i \leq q+1$ such that
$$\epsilon_{i,1}\dots,\epsilon_{i,k} \not=0 \mbox{ and } \epsilon_{i,k+1},\dots,\epsilon_{i,n}=0 \mbox{ for } \sum_{i=0}^{k-1}m_i+1 \leq i \leq \sum_{i=0}^{k-1}m_i+m_k$$
for $k=0,\dots,n-1$ and $\epsilon_{i,j} \not=0$ for all $j$ for $\sum_{i=0}^{n-1}m_i +1\leq i \leq q+1$. Further, if $\epsilon_{i,j}\not=0$ then $\epsilon_{i,j-1} = (\epsilon_{i,j})^p$. Then
$$\frac{|\{(f_{\vec{\alpha}})\in\hat{\F}_{[\vec{d}(\vec{\alpha})]} : \chi_{p^j}(F_{p^j}(x_i)) = \epsilon_{i,j}, 1\leq j\leq n, 1 \leq i \leq q+1\}|}{|\hat{\F}_{[\vec{d}(\vec{\alpha})]}|}$$
$$ =  \prod_{k=0}^{n-1}\left( \frac{(p-1)p^{2n-2k-1}}{q} \right)^{m_k}\left(\frac{q} {p^n(q+p^n-1)}\right)^\ell\left(1 + O\left(q^{-\frac{\min_{\vec{\alpha}\in\R}d(\vec{\alpha})}{2}}\right)\right).$$

\end{cor}

\begin{proof}
\textbf{Case 1:} $\epsilon_{q+1,k}\not=0$ for all $1\leq k \leq n$.

In this case we get that $(f_{\vec{\alpha}})\in\hat{\F}_{\vec{d}(\vec{\alpha})}$. Further the leading coefficient of $f_1$ under $\chi_{p^n}$ must be $\epsilon_{q+1,n}$. Thus

$$\frac{q-1}{p^n}|\{(f_{\vec{\alpha}})\in{\F}_{\vec{d}(\vec{\alpha})} : \chi_{p^j}(F_{p^j}(x_i)) = \epsilon_{i,j}, 1\leq j\leq n, 1 \leq i \leq q\}|$$
$$ = \frac{q-1}{p^n}\frac{L_{p^n-2} q^{\sum_{\vec{\alpha}\in\R}d(\vec{\alpha})}}{\zeta_q(2)^{p^n-1}} \prod_{k=0}^{n-1}\left( \frac{(p-1)p^{2n-2k-1}}{q} \right)^{m_k}\left(\frac{q} {p^n(q+p^n-1)}\right)^q$$
$$ = \frac{(q-1)(q+p^n-1)}{q}\frac{L_{p^n-2} q^{\sum_{\vec{\alpha}\in\R}d(\vec{\alpha})}}{\zeta_q(2)^{p^n-1}} \prod_{k=0}^{n-1}\left( \frac{(p-1)p^{2n-2k-1}}{q} \right)^{m_k}\left(\frac{q} {p^n(q+p^n-1)}\right)^{q+1}.$$

\textbf{Case 2:} $\epsilon_{q+1,1},\dots,\epsilon_{q+1,k} \not=0$ and $\epsilon_{q+1,k+1},\dots,\epsilon_{q+1,n}=0$ for some $k$.

In this case we get $(f_{\vec{\alpha}})\in\hat{\F}^{\vec{\beta}}_{(d(\vec{\alpha}))}$ where $\vec{\beta}=(\beta_1,\dots,\beta_n) \in\R$ such that $\beta_j=0$ for all $j\leq k$ and $\beta_{k+1}\not=0$. There are $(p-1)p^{n-k-1}$ such $\vec{\beta}$. Further $m_k$ will go to $m_k-1$ and the leading coefficient of $f_1$ under $\chi_{p^k}$ must be $\epsilon_{q+1}^k$. Then,

\begin{align*}
& \frac{q-1}{p^k} \sum_{\vec{\beta}}|\{(f_{\vec{\alpha}})\in{\F}^{\vec{\beta}}_{\vec{d}(\vec{\alpha})} : \chi_{p^j}(F_{p^j}(x_i)) = \epsilon_{i,j}, 1\leq j\leq n, 1 \leq i \leq q\}| \\
= & \frac{q-1}{p^k}(p-1)p^{n-k-1}\frac{L_{p^n-2} q^{\sum_{\vec{\alpha}\in\R}d(\vec{\alpha})-1}}{\zeta_q(2)^{p^n-1}} \prod_{j=0}^{n-1}\left( \frac{(p-1)p^{2n-2j-1}}{q} \right)^{m_j}\\
& \times \left(\frac{(p-1)p^{2n-2k-1}}{q} \right)^{-1}\left(\frac{q} {p^n(q+p^n-1)}\right)^q \left(1 + O\left(q^{-\frac{\min_{\vec{\alpha}\in\R}d(\vec{\alpha})}{2}}\right)\right) \\
= & \frac{(q-1)(q+p^n-1)}{q}\frac{L_{p^n-2} q^{\sum_{\vec{\alpha}\in\R}d(\vec{\alpha})}}{\zeta_q(2)^{p^n-1}} \prod_{j=0}^{n-1}\left( \frac{(p-1)p^{2n-2j-1}}{q} \right)^{m_j}\\
& \times\left(\frac{q} {p^n(q+p^n-1)}\right)^{q+1} \left(1 + O\left(q^{-\frac{\min_{\vec{\alpha}\in\R}d(\vec{\alpha})} {2}}\right)\right).
\end{align*}

Thus, independent of the behavior at $x_{q+1}$, we get out result.

\end{proof}

Which leads us to restate and prove Theorem \ref{mainthm} for $r=p^n$.

\begin{thm}\label{primepowerthm}

Let $M_j\in\mathbb{Z}[\zeta_{p^j}]$ for $j=1,\dots,n$. Then
$$\frac{|\{(f_{\vec{\alpha}})\in\hat{\F}_{[\vec{d}(\vec{\alpha})]} : S_{p^j}(F_{p^j}) = M_j, 1\leq j\leq n\}|}{|\hat{\F}_{[\vec{d}(\vec{\alpha})]}|}$$
$$= \Prob\left( \sum_i X_{i,1} = M_1 \mbox{ and } \dots \mbox{ and } \sum_i X_{i,n} = M_n \right) \left(1 + O\left(q^{-\frac{\min_{\vec{\alpha}\in\R}d(\vec{\alpha})}{2}}\right)\right)$$
where the $X_{i,j}$ are random variables that take values in $\mu_{p^j}\cup\{0\}$ such that $X_{i,j}$ and $X_{h,k}$ are i.i.d unless $i=h$ and
\begin{align*}
&\Prob(X_{i,j}=0) = \frac{p^n-p^{n-j}}{q+p^n-1} & \Prob(X_{i,j} = \epsilon_{i,j}) = \frac{q+p^{n-j}-1}{p^j(q+p^n-1)}\\
&\Prob(X_{i,j+1}=0 | X_{i,j} =0)=1 & \Prob(X_{i,j-1}=(X_{i,j})^p | X_{i,j} \not =0)=1
\end{align*}
\begin{align*}
&\Prob(X_{i,1} = \epsilon_{i,1},\dots,X_{i,j}=\epsilon_{i,j} \mbox{ and } X_{i,j+1}, \dots, X_{i,n} =0) = \begin{cases}\frac{(p-1)p^{n-2j-1}}{q+p^n-1} & \epsilon_{i,k-1} = (\epsilon_{i,k})^p, k=2,\dots,j \\ 0 & \mbox{otherwise} \end{cases}\\
&\Prob(X_{i,1}=\epsilon_{i,1},\dots,X_{i,n} =\epsilon_{i,n}) = \begin{cases} \frac{q}{p^n(q+p^n-1)} &\epsilon_{i,k-1} = (\epsilon_{i,k})^p, k=2,\dots,n \\ 0 & \mbox{otherwise} \end{cases}
\end{align*}
\end{thm}

\begin{proof}

$$\frac{|\{(f_{\vec{\alpha}})\in\hat{\F}_{[\vec{d}(\vec{\alpha})]} : S_{p^j}(F_{p^j}) = M_j, 1\leq j\leq n\}|}{|\hat{\F}_{[\vec{d}(\vec{\alpha})]}|}$$
$$ = \sum_{\substack{(E_{1,j},\dots,E_{q+1,j}) \in \mu_{p^j} \cup \{0\} \\ \sum_i E_{i,j} = M_j \\ E_{i,j}=0 \implies E_{i,j+1} = 0 \\ E_{i,j} \not=0 \implies E_{i,j-1} = (E_{i,j})^p \\ j=1,\dots,n }} \prod_{k=0}^{n-1}\left( \frac{(p-1)p^{2n-2k-1}}{q} \right)^{m_k}\left(\frac{q} {p^n(q+p^n-1)}\right)^{q+1}\left(1 + O\left(q^{-\frac{\min_{\vec{\alpha}\in\R}d(\vec{\alpha})}{2}}\right)\right)$$
$$ = \sum_{\substack{(E_{1,j},\dots,E_{q+1,j}) \in \mu_{p^j} \cup \{0\} \\ \sum_i E_{i,j} = M_j \\ E_{i,j}=0 \implies E_{i,j+1} = 0 \\ E_{i,j} \not=0 \implies E_{i,j-1} = (E_{i,j})^p \\ j=1,\dots,n }} \prod_{k=0}^{n-1}\left( \frac{(p-1)p^{n-2k-1}}{q+p^n-1} \right)^{m_k}\left(\frac{q} {p^n(q+p^n-1)}\right)^{q+1-\sum m_k}\left(1 + O\left(q^{-\frac{\min_{\vec{\alpha}\in\R}d(\vec{\alpha})}{2}}\right)\right) $$
$$ = \mbox{Prob}\left( \sum_i X_{i,1} = M_1 \mbox{ and } \dots \mbox{ and } \sum_i X_{i,n} = M_n \right)\left(1 + O\left(q^{-\frac{\min_{\vec{\alpha}\in\R}d(\vec{\alpha})}{2}}\right)\right)$$
where the $X_i$ have the desired properties.
\end{proof}

\section{General $r$} \label{generalr}

Suppose now that $r=r_1r_2$ with $(r_1,r_2)=1$. Then, relabeling the $f_i$ as $f_{i,j}$ we can write
$$F(X) = \prod_{\substack{i=0,\dots,r_1-1\\j=0,\dots,r_2-1\\(i,j)\not=(0,0)}}f_{i,j}(X)^{s_2r_2i+s_1r_1j \mod{r}}$$
where $s_1 \equiv r_1^{-1} \mod{r_2}$ and $s_2 \equiv r_2^{-1}\mod{r_1}$. With this notation we now get
\begin{align*}
& F_{r_1} = \prod_{\substack{i=0,\dots,r_1-1\\j=0,\dots,r_2-1\\(i,j)\not=(0,0)}} f_{i,j}^{i \mod{r_1}} & F_{r_2} = \prod_{\substack{i=0,\dots,r_1-1\\j=0,\dots,r_2-1\\(i,j)\not=(0,0)}} f_{i,j}^{j \mod{r_2}}.
\end{align*}
Therefore
$$F_{r_1}(X)^{s_2r_2}F_{r_2}(X)^{s_1r_1} = F(X) \left(\prod_{i,j}f_{i,j}(X)^{n_{i,j}}\right)^r$$
for some, potentially $0$, exponents $n_{i,j}$. Further we see that $F(x)=0$ if and only if $F_{r_1}F_{r_2}(x)=0$ as all the factors that appear in $F$ appear in $F_{r_1}F_{r_2}$. Therefore the values of $F_d(x)$ for all $d|r$ are determined by the values of $F_{p^n}(x)$ where $p$ is a prime such that $p^n|r$. That is,
$$|\{F \in \F_{(d_1,\dots,d_r}) : \chi_d(F_{d}(x_i)) = \epsilon_{d,i}, \mbox{ for all } d|r, 1\leq i \leq \ell\}|$$
$$ = |\{F \in \F_{(d_1,\dots,d_r}) : \chi_d(F_{d}(x_i)) = \epsilon_{d,i}, d=p^n|r, p \mbox{ a prime}, 1\leq i \leq \ell\}|$$

Now, suppose $r=p_1^{t_1}\cdots p_n^{t_n}$. Define
$$\R' = [0,\dots,p_1^{t_1}-1] \times \dots \times [0,\dots,p_n^{t_n}-1] \setminus\{(0,\dots,0)\}$$
Let
$$\phi:\R' \to [1,\dots,r-1] $$
be the isomorphism that comes from the Chinese Remainder Theorem. Then if $F=\prod_{i=1}^{r-1}f_i^i$, we can relabel the $f_i$ to get
$$F = \prod_{\vec{\beta}\in\R'} f_{\vec{\beta}}^{\phi(\vec{\beta})}.$$

Notice that $\phi(\vec{\beta}) \equiv k \mod{p_j^{t_j}}$ if and only if $\beta_j =k$. Therefore
$$F_{(p_j^{t_j})} = \prod_{\vec{\beta}\in\R'} f_{\vec{\beta}}^{\phi(\vec{\beta}) \mod{p_j^{t_j}}} = \prod_{k=0}^{p_j^{t_j}-1} \prod_{\substack{\beta\in\R' \\ \beta_j = k}} f_{\vec{\beta}}^k  = \prod_{\vec{\beta}\in\R'}f_{\vec{\beta}}^{\beta_j}.$$

However, we need all powers of the primes. Therefore we define

$$\R = [0,\dots,p_1-1]^{t_1}\times\dots\times[0,\dots,p_n-1]^{t_n} \setminus \{(0,\dots,0)\}.$$
Let $T_j = \sum_{i=1}^{j} t_i$. As usual, for $\vec{\alpha}\in\R$ we write it as $\vec{\alpha} =(\alpha_1,\dots,\alpha_{T_n})$. Then there is an isomorphism $\psi:\R \to\R'$ such that
$$ \psi(\vec{\alpha}) = (\alpha_1 + p_1\alpha_2 + \dots+p_1^{t_1-1}\alpha_{T_1}, \dots, \alpha_{T_{n-1}+1} + p_n \alpha_{T_{n-1}+2} + \dots + p_n^{t_n-1}\alpha_{T_n})$$
Therefore, if we relabel the $f_{\vec{\beta}}$ we get
$$F_{(p_j^{t_j})}(X) = \prod_{\vec{\alpha}\in\R} f_{\vec{\alpha}}^{\alpha_{T_{j-1}+1} + p\alpha_{T_{j-1}+2} + \dots p^{t_j-1}\alpha_{T_j}}(X)$$

Moreover, for any $1\leq k_j \leq t_j$
$$F_{(p_j^{k_j})}(X) = \prod_{\vec{\alpha}\in\R} f_{\vec{\alpha}}^{\alpha_{T_{j-1}+1} + p\alpha_{T_{j-1}+2} + \dots p^{k_j-1}\alpha_{T_{j-1}+k_j}}(X)$$

Then we get
$$|\{F \in \F_{(d_1,\dots,d_{r-1})} : F_{(p_j^{k_j})}(x_i) = a_{i,j,k_j}, 1 \leq j \leq n, 1 \leq k_j \leq t_j, 1 \leq i \leq \ell\}|$$
$$ = |\{(f_{\vec{\alpha}})\in \F_{\vec{d}(\vec{\alpha})} : F_{(p_j^{k_j})}(x_i)  = a_{i,j,k_j}, 1 \leq j \leq n, 1 \leq k_j \leq t_j, 1 \leq i \leq \ell\}|$$

where $d(\vec{\alpha}) = d_{\phi\circ\psi(\vec{\alpha})}$.

\begin{prop}\label{generalprop}
Let $a_{i,j,k_j}\in\Ff^*_q$ such that $a_{i,j,k_j}=a_{i,j,k_j-1}(b_{i,j,k_j})^{p^{k_j-1}}$ for some $b_{i,j,k_j}\in\Ff^*_q$. Further we let that $a_{i,j,0}=1$ so that $a_{i,j,1} = b_{i,j,1}$.
\begin{align*}
&|\{(f_{\vec{\alpha}})\in\F_{\vec{d}(\vec{\alpha})} : F_{(p_j^{k_j})}(x_i) = a_{i,j,k_j}, 1\leq j \leq n, 1\leq k_j \leq t_j, 1\leq i \leq \ell\}|\\
=&\frac{L_{r-2} q^{\sum_{\vec{\alpha}\in\R}d(\vec{\alpha})}}{\zeta_q(2)^{r-1}} \left(\frac{q \prod_{j=1}^n p^{\frac{t_j(t_j-1)}{2}}}{ (q-1)^{T_n}(q+r-1)}\right)^\ell\left(1 + O\left(q^{-\frac{\min_{\vec{\alpha}\in\R}d(\vec{\alpha})}{2}}\right)\right).
\end{align*}
\end{prop}

\begin{proof}
As we saw in the proof of Proposition \ref{primepowerprop}, it will be enough to consider what values the polynomials
$$F_{(j,k_j)} = \prod_{\vec{\alpha}\in\R} f_{\vec{\alpha}}^{\alpha_{T_{j-1}+k_j}}$$
up to some root of unity. That is,
$$|\{(f_{\vec{\alpha}})\in\F_{\vec{d}(\vec{\alpha})} : F_{(p_j^{k_j})}(x_i) = a_{i,j,k_j}, 1\leq j \leq n, 1\leq k_j \leq t_j, 1\leq i \leq \ell\}|$$
$$ = |\{(f_{\vec{\alpha}})\in\F_{\vec{d}(\vec{\alpha})} : F_{(j,k_j)}(x_i) = \epsilon_{i,j,k_j}b_{i,j,k_j}, \epsilon_{i,j,k_j} \in \mu_{p^{k_j-1}},1\leq j \leq n, 1\leq k_j \leq t_j, 1\leq i \leq \ell\}|.$$

We can define an isomorphism
$$\phi: \{(j,k_j): 1\leq j \leq n, 1 \leq k_j \leq T_j\} \to \{1,\dots,T_n\}$$
by
$$\phi(j,k_j) = T_{j-1}+k_j.$$
Let $\epsilon_{i,j} = \epsilon_{i,\phi^{-1}(j)}$, $\beta_{i,j} = \beta_{i, \phi^{-1}(j)}$ and $F_j = F_{\phi^{-1}(j)}$, then
$$F_j = \prod_{\vec{\alpha}\in\R}f_{\vec{\alpha}}^{\alpha_j}.$$
If we denote $\mu_{m} = \mu_{p_j^{k_j}}$ where $\phi(j,k_j)=m$ then,

$$|\{(f_{\vec{\alpha}})\in\F_{\vec{d}(\vec{\alpha})} : F_{(j,k_j)}(x_i) = \epsilon_{i,j,k_j}b_{i,j,k_j}, \epsilon_{i,j,k_j} \in \mu_{p_j^{k_j-1}},1\leq j \leq n, 1\leq k_j \leq t_j, 1\leq i \leq \ell\}|$$
$$ = |\{(f_{\vec{\alpha}})\in\F_{\vec{d}(\vec{\alpha})} : F_{j}(x_i) = \epsilon_{i,j}b_{i,j}, \epsilon_{i,j} \in \mu_{j},1\leq j \leq T_n, 1\leq i \leq \ell\}|$$
$$ = \frac{L_{r-2} q^{\sum_{\vec{\alpha}\in\R}d(\vec{\alpha})}}{\zeta_q(2)^{r-1}} \left(\frac{q \prod_{j=1}^n p^{\frac{t_j(t_j-1)}{2}}}{ (q-1)^{T_n}(q+r-1)}\right)^\ell \left(1 + O\left(q^{-\frac{\min_{\vec{\alpha}\in\R}d(\vec{\alpha})}{2}}\right)\right)$$
where the last equality comes from Proposition \ref{setcount}.

\end{proof}

Define
$$\s = [0,\dots,t_1]\times\dots\times[0,\dots,t_n]\setminus{(t_1,\dots,t_n)}.$$
For any $S\in\s$ we will write $S = (s_1,\dots,s_n)$.  We will use $\s$ to count the number of $1\leq i\leq \ell$ such that $F_{(p_j^{k_j})}(x_i)=0$. That is, for every $S=(s_1,\dots,s_n)\in\s$ we say that for some $i$ $F_{(p_j^{s_j+1})}(x_i)=0$ (and hence $F_{(p_j^{k_j})}(x_i)=0$ for all $k_j >s_j$). If $s_j=t_j$ this will correspond to the case that $F_{(p_j^{k_j})}(x_i)\not=0$ for all $1\leq i\leq \ell$ for all $k_j$, hence why we exclude the element $(t_1,\dots,t_n)$ from $\s$ as this corresponds to the case that there are no zeros, which we treated in Proposition \ref{generalprop}. With this motivation we will define $J_{\s} = \{j : s_j \not = t_j\}$.

\begin{cor}\label{generalcor}
Let $a_{i,j,k_j}\in\Ff_q$ such that for $m_S$ of the $i$ we have
$$a_{i,j,1}, \dots, a_{i,j,s_j}\not=0 \mbox{ and } a_{i,j,s_j+1},\dots,a_{i,j,t_j} =0,\quad 1\leq j \leq n$$
for all $S\in\s$. Further if $a_{i,j,k_j}\not=0$ then $a_{i,j,k_j}=a_{i,j,k_j-1}(b_{i,j,k_j})^{p^{k_j-1}}$ for some $b_{i,j,k_j}\in\Ff^*_q$. Again, we set $a_{i,j,0}=1$ so that $a_{i,j,1} = b_{i,j,1}$. Then,
\begin{align*}
&|\{(f_{\vec{\alpha}})\in\F_{\vec{d}(\vec{\alpha})} : F_{(p_j^{k_j})}(x_i) = a_{i,j,k_j}, 1\leq j \leq n, 1\leq k_j \leq t_j, 1\leq i \leq \ell\}|\\
=& \frac{L_{r-2} q^{\sum_{\vec{\alpha}\in\R}d(\vec{\alpha})}}{\zeta_q(2)^{r-1}} \prod_{S\in\s} \left( \frac{\prod_{j\in J_S} (p_j-1)p_j^{t_j-s_j} \prod_{j=1}^n p_j^{\frac{s_j(s_j-1)}{2}}}{(q+r-1)\prod_{j=1}^n (q-1)^{s_j}} \right)^{m_S}\\
& \times \left(\frac{q \prod_{j=1}^n p_j^{\frac{t_j(t_j-1)}{2}}}{ (q-1)^{T_n}(q+r-1)}\right)^{\ell-\sum_{S\in\s}m_S} \left(1 + O\left(q^{-\frac{\min_{\vec{\alpha}\in\R}d(\vec{\alpha})}{2}}\right)\right).
\end{align*}
\end{cor}

\begin{proof}

If we replace $f_{\vec{\alpha}}$ by $f'_{\vec{\alpha}}$, the polynomial divided by its roots, and $F'_{(p_j^{k_j})}$ as the corresponding product of $f'_{\vec{\alpha}}$. $F'_{(p_j^{k_j})}(x_i)$ will be determined by $F_{(p_j^{k_j})}(x_i)$ for $k_j \leq s_j$ and $F'_{(p_j^{k_j})}(x_i)$ will be determined, up to a $p^{k_j-1}$th root of unity, by $F_{(p_j^{k_j-1})}(x_i)$ for $s_j < k_j \leq t_j$. Summing up over all the necessary partitions of $m_S$ will give the desired result.

\end{proof}

Note that if we let $d=\prod_{j=1}^n p_j^{s_j}$, then we can write
$$\prod_{j\in J_S} (p_j-1)p_j^{t_j-s_j} = \phi\left(\frac{r}{d}\right).$$
This illustrates why it was important to consider the set $J_S$ for if the left hand product was over all the $j$, we would not get this nice equality.

Let $\vec{\textbf{1}}\in\R$ be the element that has $1$ in the $T_j+1$ position $j=0,\dots,n-1$ and $0$ everywhere else. Then set $\R' = \R \setminus\vec{\textbf{1}}$. Define

$$\hat{\F}_{\vec{d}(\vec{\alpha})} = \{(f_{\vec{\textbf{1}}},(f_{\vec{\alpha}})) \in \hat{\F}_{d_{\vec{\textbf{1}}}} \times \prod_{\vec{\alpha}\in\R'}\F_{\vec{\alpha}} : (f_{\vec{\alpha}},f_{\vec{\beta}})=1 \mbox{ for all } \vec{\alpha}\not=\vec{\beta}\}$$

With this definition, define $\F^{\vec{\beta}}_{\vec{d}(\vec{\alpha})},\F_{[\vec{d}(\vec{\alpha})]}$ and $\hat{\F}_{[\vec{d}(\vec{\alpha})]}$ the same way as in Section \ref{primepower}.
%To incorporate the prime at infinity we now need to extend this definition. Let
%$$\delta_{A,B} = \begin{cases} 1 & A = B \\ 0 & A \not = B \end{cases}$$
%Then define
%
%$$\F^{B}_{(d_A)} = \F_{(d_A - \delta_{A,B})}$$
%Likewise
%$$\hat{\F}^{B}_{(d_{A})} = \hat{\F}_{(d_{A} - \delta_{A,B})}$$
%Then,
%$$\F_{[d_A]} = \F_{(d_A)} \cup \bigcup_{B\in\R} {\F}^{B}_{(d_A)}$$
%$$\hat{\F}_{[d_A]} = \hat{\F}_{(d_A)}\cup\bigcup_{B\in\R} \hat{\F}^{B}_{(d_A)}$$

Now, if $x_{q+1}$ is the point at infinity, then

$$F_{(p_j^{k_j})}(x_{q+1}) = \begin{cases} \mbox{leading coefficient of } f_{\vec{\textbf{1}}} & (f_{\vec{\alpha}})\in\hat{\F}^{\vec{\beta}}_ {d(\vec{\alpha})}\\ 0 & \mbox{otherwise}  \end{cases}$$
where $\beta = (\beta_{1},\dots,\beta_{T_n})\in\R$ is any tuple such that $\beta_{T_j+1},\dots,\beta_{T_j+k_j}=0$.

Similarly to Corollaries \ref{primepowercor2} and \ref{primepowercor3}, we get

\begin{cor}\label{gencor2}
Let $\epsilon_{i,j,k_j}\in\mu_{p_j^{k_j}}\cup\{0\}$ such that for $m_S$ of the $i$ we have
$$\epsilon_{i,j,1}, \dots, \epsilon_{i,j,s_j}\not=0 \mbox{ and } \epsilon_{i,j,s_j+1},\dots,\epsilon_{i,j,t_j} =0,\quad 1\leq j \leq n$$
for all $S\in\s$. Further if $\epsilon_{i,j,k_j}\not=0$ then $\epsilon_{i,j,k_j}=(\epsilon_{i,j,k_j-1})^p$.
\begin{align*}
&\frac{|\{(f_{\vec{\alpha}})\in\hat{F}_{[\vec{d}(\vec{\alpha})]} : \chi_{p_j^{k_j}}(F_{(p_j^{k_j})}(x_i)) = \epsilon_{i,j,k_j}, 1\leq j \leq n, 1\leq k_j \leq t_j, 1\leq i \leq q+1\}|}{|\hat{F}_{[\vec{d}(\vec{\alpha})]}|}\\
= & \prod_{S\in\s}  \left( \frac{\prod_{j\in J_{\s}}(p_j-1)p_j^{t_j-s_j}}{\prod_{j=1}^n p_j^{s_j} (q+r-1)}\right)^{m_S} \left(\frac{q}{r(q+r-1)}\right)^{q+1-\sum_{S\in\s}m_S}\left(1 + O\left(q^{-\frac{\min_{\vec{\alpha}\in\R}d(\vec{\alpha})} {2}}\right)\right).
\end{align*}

\end{cor}

Applying this Corollary we obtain

\begin{proof}[Proof of Theorem \ref{mainthm}]

$$\frac{|\{(f_{\vec{\alpha}})\in\hat{F}_{[\vec{d}(\vec{\alpha})]} : S_d(F_{(d)}) = M_d, \forall d|r \}|}{|\hat{F}_{[\vec{d}(\vec{\alpha})]}|}$$
$$ = \sum_{\substack{E_{d,1},\dots,E_{d,q+1} \in \mu_d \cup\{0\} \\ \sum_{i=1}^{q+1}E_{d,i}=M_d \\ E_{d,i} = \prod_{p|d}(E_{p^{v_p(d)},i})^{\sigma_p} \forall d|r  \\ E_{p_j^{k_j},i}=0 \implies E_{p_j^{s_j},i}=0 \mbox{ and} \\ E_{p_j^{s_j},i}\not=0 \implies E_{p_j^{k_j},i}=E_{p_j^{s_j},i}^{p_j^{s_j-k_j}} \forall j, k_j\leq s_j\leq t_j }} \frac{|\{(f_{\vec{\alpha}})\in\hat{F}_{[\vec{d}(\vec{\alpha})]} : \chi_d(F_{(d)}(x_i)) =E_{d,i}, \forall d|r, 1\leq i \leq q+1 \}|}{|\hat{F}_{[\vec{d}(\vec{\alpha})]}|}$$
\begin{align*}
= \sum_{\substack{E_{d,1},\dots,E_{d,q+1} \in \mu_d \cup\{0\} \\ \sum_{i=1}^{q+1}E_{d,i}=M_d \\ E_{d,i} = \prod_{p|d}(E_{p^{v_p(d)},i})^{\sigma_p} \forall d|r  \\ E_{p_j^{k_j},i}=0 \implies E_{p_j^{s_j},i}=0 \mbox{ and}\\ E_{p_j^{s_j},i}\not=0 \implies E_{p_j^{k_j},i}=E_{p_j^{s_j},i}^{p_j^{s_j-k_j}} \forall j, k_j\leq s_j\leq t_j }} & \prod_{S\in\s}  \left( \frac{\prod_{j\in J_{\s}}(p_j-1)p_j^{t_j-s_j}}{\prod_{j=1}^n p_j^{s_j} (q+r-1)}\right)^{m_S}  \\
& \times \left(\frac{q}{r(q+r-1)}\right)^{q+1-\sum_{S\in\s}m_S}\left(1 + O\left(q^{-\frac{\min_{\vec{\alpha}\in\R}d(\vec{\alpha})}{2}}\right)\right)
\end{align*}
$$ = \Prob(\sum_{i=1}^{q+1} X_{d,i} = M_d \mbox{ for all } d|r)\left(1 + O\left(q^{-\frac{\min_{\vec{\alpha}\in\R}d(\vec{\alpha})}{2}}\right)\right)$$

Where in the subscript we have $\sigma_p = p^{-v_p(d)} \mod{\frac{d}{p^{v_p(d)}}}$.

Further, if we let $d=\prod_{j=1}^n p_j^{s_j}$ then factor being raised to the $m_S$ can be rewritten as
$$\frac{\phi\left(\frac{r}{d}\right)}{d(q+r-1)}$$
and, if $d\not=r$, the $X_{p^s,i}$ satisfy the relationship
$$\Prob\left(X_{p^s,i} = \epsilon_{p^s,i}\not=0,1 \leq s \leq v_p(d) \mbox{ and } X_{p^s,i} =0, v_p(d)<s\leq v_p(r) \mbox{ for all } p|r\right)$$
$$ = \begin{cases}\frac{\phi(\frac{r}{d})}{d(q+r-1)} & \mbox{if }  \epsilon_{p^{s-1},i} = \epsilon_{p^s,i}^p \mbox{ for all } p|r, 1\leq s \leq v_p(d)  \\ 0  & \mbox{otherwise} \end{cases}.$$
If $d=r$, then
$$\Prob\left(X_{p^s,i} = \epsilon_{p^s,i}, s \leq v_p(r), \mbox{ for all } p|r\right) = \begin{cases} \frac{q }{ r(q+r-1)} & \mbox{if }  \epsilon_{p^{s-1},i} = \epsilon_{p^s,i}^p, 1\leq  s \leq v_p(r), \mbox{ for all } p|r \\ 0 & \mbox{otherwise} \end{cases}.$$

Notice also, since the $\sigma_p$ come from the Chinese Remainder Theorem, if $X_{d,i}\not=0$ then the value of $X_{p^{v_p(d)},i}$ is uniquely determined and hence so are $X_{p^k,i}$ for all $p|r$ and $k \leq v_p(d)$. Moreover, for $k> v_p(d)$, $X_{p^k,i}$ will either be $0$ or determined, up to a $p^{k-v_p(d)}$th root of unity. From here we get

\begin{align*}
&\Prob\left(X_{d,i}=\epsilon_{d,i} \not=0\right)\\
&= \sum_{\substack{d_1 \\ d|d_1}} \sum {}^{'} \Prob\left(X_{p^s,i}=\epsilon_{p^{v_p(d_1)},i}^{p^{v_p(d_1)-s}},1\leq s \leq v_p(d_1) \mbox{ and } X_{p^s,i} =0, v_p(d_1)<s\leq v_p(r) \mbox{ for all } p|r \right)
\end{align*}

where $\sum {}^{'}$ is the sum that runs over all 
$$\epsilon_{p^{v_p(d_1)},i} \in \mu_{p^{v_p(d_1)}} \mbox{ such that } \epsilon_{p^{v_p(d_1)},i}^{p^{v_p(d_1-v_p(d)}} = \epsilon_{p^{v_p(d)},i} \mbox{ for all } p|r.$$

Hence,
\begin{align*}
& = \sum_{\substack{d_1\not=r \\ d|d_1}} \sum{}^{'}  \frac{\phi\left(\frac{r}{d_1}\right)}{d_1(q+r-1)}  +   \sum_{d_1=r}\sum{}^{'}  \frac{q}{r(q+r-1)}\\
& = \frac{q+\sum_{\substack{d_1\not=r \\ d|d_1 }}\phi\left(\frac{r}{d_1}\right)  }{d(q+r-1)} = \frac{q + \sum_{n|\frac{r}{d}}\phi\left(\frac{r/d}{n}\right)-1}{d(q+r-1)} = \frac{q+\frac{r}{d}-1}{d(q+r-1)}.
\end{align*}

%\sum_{\substack{\epsilon_{p^{v_{p}(d_1)},i}\in \mu_{p^{v_p(d_1)}} \\ \epsilon_{p^{v_p(d_1)},i}^{p^{v_p(d_1)-v_p(d)}} = \epsilon_{p^{v_p(d)},i} }}

Therefore
$$\Prob(X_{d,i}=0) = 1 - \sum_{\epsilon_{d,i}\in\mu_d} \Prob(X_{d,i}=\epsilon_{d,i}) = 1- \frac{q+\frac{r}{d}-1}{q+r-1} = \frac{r-\frac{r}{d}}{q+r-1}$$

\end{proof}

\section{Heuristic}\label{heuristic}

In this section we will discuss a heuristic for Corollary \ref{generalcor} and consequently Theorem \ref{mainthm}. First we will need Lemma 8.1 from \cite{BDFL2}.

\begin{lem}

Let $S_n$ be the set of $n$-tuples $(F_1,\dots,F_n)$ of nonzero residues of modulo $(X-t)^2$ such that $(X-t)$ divides at most one of the $F_i$. Then
$$|S_n| = q^{n-1}(q-1)^n(q+n).$$

\end{lem}

The set $S$ models the set of $n$ squarefree coprime polynomials. Write $F=f_1f_2^2\dots f_{r-1}^{r-1}$ with $(f_1,\dots,f_{r-1})\in S_{r-1}$. If $r=\prod_{j=1}^n p_j^{t_j}$ then Corollary \ref{generalcor} deals with $F_{(p_j^{k_j})}$ for $1\leq j \leq n$, $1\leq k_j \leq t_j$. That is, we want to determine how many $(r-1)$-tuples there are in $S_{r-1}$ that satisfy
$$F_{(p_j^{k_j})} \equiv a_{j,k_j} \mod{(X-t)^2}$$
where $a_{j,k_j} = a_{j,k_j-1}(b_{j,k_j})^{p_j^{k_j-1}}$ for some $b_{j,k_j}$.

Suppose $a_{j,t_j}\not\equiv0 \mod{(X-t)}$ for all $j$. Using the notation $F_{(p_j^{k_j})} = \prod_{i=1}^{r-1} f_i^{i \mod{p_j^{k_j}}}$ we will start with $j=1$. The condition of
$$F_{(p_1)} \equiv a_{1,1} \mod{(X-t)^2}$$
says that there are $(q(q-1))^{r-\frac{r}{p_1}-1}$ choices for $f_i$ with $i\not=1$ and $q$ choices for $f_1$ which gives a total of $q^{r-\frac{r}{p_1}}(q-1)^{r-\frac{r}{p_1}-1}$ choices of the $f_i$. Now if we fix one of these choices and look at
$$F_{(p_1^2)} \equiv a_{1,2} = a_{1,1}(b_{1,2})^p \mod{(X-t)^2}$$
then there would be $(q(q-1))^{\frac{r}{p_1}-\frac{r}{p_1^2}-1}$ choices for the $f_{p_1i}$ with $i\not=1$ and $p_1q$ choices for $f_{p_1}$ which gives a total of $p_1q^{r-\frac{r}{p_1^2}}(q-1)^{r-\frac{r}{p_1^2}-2}$. From here we can see that if we consider all the conditions
$$F_{(p_1^{k_1})} \equiv a_{1,k_1} \mod{(X-t)^2} \quad \quad \quad 1\leq k_1 \leq t_1$$
we will get $p_1^{\frac{t_1(t_1-1)}{2}}q^{r-\frac{r}{p_1^{t_1}}}(q-1)^{r-\frac{r}{p_1^{t_1}}-t_1}$. Now, if we fix these choices and look at
$$F_{(p_2)} \equiv a_{2,1} \mod{(X-t)^2}$$
then all the $f_i$ appearing in $F_{(p_2)}$ such that $i \not \equiv 0 \mod{p_1^{t_1}}$ are already determined. If we let $j = p_1^{-t_1} \mod{\frac{r}{p_1^{t_1}}}$ then there are $(q(q-1))^{\frac{r}{p_1^{t-1}}-\frac{r}{p_1^{t_1}p_2}-1}$ choices for the $f_{p_1^{t_1}i}$ such that $i\not=j$, $i\not\equiv 0 \mod{p_2}$ and $q$ choices for $f_{p_1^{t_1}j}$. Therefore we get a total of $p_1^{\frac{t_1(t_1-1)}{2}}q^{r-\frac{r}{p_1^{t_1}p_2}}(q-1)^{r-\frac{r}{p_1^{t_1}p_2}-t_1-1}$. From here it is clear what will happen and if we look at all the conditions
$$F_{(p_j^{k_j})} \equiv a_{j,k_j} \mod{(X-t)^2} \quad \quad \quad 1 \leq k_j \leq t_j, 1 \leq j \leq n$$
then we get
$$(q(q-1))^{r-1}\prod_{j=1}^n \frac{p_j^{\frac{t_j(t_j-1)}{2}}}{(q-1)^{t_j}}.$$
Dividing by $|S|$, we get
$$\frac{q}{q+r-1} \prod_{j=1}^n \frac{p^{\frac{t_j(t_j-1)}{2}}}{(q-1)^{t_j}}$$
which is consistent with Proposition \ref{generalprop}.

Now suppose that not all $a_{i,j,k_j}$ are non-zero. As with the statement of Corollary \ref{generalcor} we have $s_1,\dots,s_n$ such that $0\leq s_j \leq t_j$ (but not all $s_j=t_{j}$) and $a_{j,k_j} \equiv 0 \mod{(X-t)}$ if and only if $k_j > s_j$, $j=1,\dots,n$. Let $d=\prod_{j=1}^n p_j^{s_j}$ then, as above, the conditions
$$F_{(p_j^{k_j})} \equiv a_{j,k_j} \mod{(X-t)^2} \quad \quad \quad 1\leq k_j \leq s_j, j=1,\dots,n$$
yields
$$(q(q-1))^{r-d} \prod_{j=1}^n \frac{p_j^{\frac{s_j(s_j-1)}{2}}}{(q-1)^{s_j}} $$
choices of the $f_i$. Fixing one of these choices, consider the conditions
$$F_{(p_j^{s_j+1})} \equiv a_{j,s_j+1} \mod{(X-t)^2} \quad \quad \quad j = 1, \dots, n \mbox{ such that } s_j \not =t_j.$$
The value of $f_i$ is already determined if $i \not\equiv 0 \mod{d}$. Further there is exactly one $1\leq i \leq r-1$ such that $v_{p_j}(i) = s_j$ for all $1\leq j\leq n$ such that $f_i \equiv 0 \mod{(X-t)}$. Then $f_i$ has $(q-1)$ different choices and the rest of the $f_j$ have $(q(q-1))^{\frac{r}{d}-\frac{r}{d'}-1}$ different choices where
$$d' = \prod_{\substack{j=1 \\ s_j\not=t_j}}^n p_j^{s_j+1} \prod_{\substack{j=1 \\ s_j=t_j}}^n p_j^{t_j}.$$
Moreover, there are $\phi(\frac{r}{d})=\prod_{j=1, s_j \not = t_j}^n (p_j-1)p_j^{t_j-s_j-1}$ such $i$ that satisfy $v_{p_j}(i)=s_j$ for all $1\leq j \leq n$. Hence with the new conditions there are
$$q^{r-\frac{r}{d'}-1}(q-1)^{r-\frac{r}{d'}}\phi\left(\frac{r}{d}\right)\prod_{j=1}^n \frac{p_j^{\frac{s_j(s_j-1)}{2}}}{(q-1)^{s_j}}$$
choices of the $f_i$.

For the remaining $\frac{r}{d'}-1$ of the $f_i$ not accounted for, the conditions
$$F_{(p_j^{k_j})} \equiv a_{j,k_j} \mod{(X-t)^2} \quad \quad \quad k_j> s_j+1, j = 1, \dots, n \mbox{ such that } s_j\not=t_j$$
are already satisfied as $a_{j,k_j}\equiv 0 \mod{(X-t)}$. Therefore, we need only $f_i\not\equiv0\mod{(X-t)}$ giving $(q(q-1))^{\frac{r}{d'}-1}$ choices. Therefore our final count will be
$$q^{r-2}(q-1)^{r-1}\phi\left(\frac{r}{d}\right)\prod_{j=1}^n \frac{p_j^{\frac{s_j(s_j-1)}{2}}}{(q-1)^{s_j}}.$$

Dividing by $|S|$ gives
$$\frac{\phi\left(\frac{r}{d}\right)}{q+r-1}\prod_{j=1}^n \frac{p_j^{\frac{s_j(s_j-1)}{2}}}{(q-1)^{s_j}}$$
which is consistent with Corollary \ref{generalcor}. 

%means that only one of $f_i\equiv 0 \mod{(X-t)}$ for some $v_{p_1}(i) = s_1-1$.  Fix an $i$, then $f_i$ has $q-1$ different choices. The other $f_j$ have $(q(q-1))^{\frac{r}{p_1^{s_1-1}}-\frac{r}{p_1^{s_1}}-1}$ choices. Further, there are $p_1-1$ different choices for the $i$ so we get a total of $(p_1-1)p_1^{\frac{(s_1-1)(s_1-2)}{2}}q^{r-\frac{r} {p_1^{s_1}}-1}(q-1)^{r-\frac{r}{p_1^{s_1}}-s_1+1}$ choices. The condition
%$$F_{p_1^{s_1+1}} \equiv a_{1,s_1+1} \mod{(X-t)^2}$$
%Then the condition imposed by $a_{1,s_1+1}\equiv 0 \mod{(X-t)}$ is already satisfied. Thus all the new $f_i$ appearing will act like in the previous case. That is, we will obtain a new factor of

\textbf{Acknowledgements:} I would like to thank Chantal David for the many discussions we had about this topic and for taking time to make sure the final paper was presentable. I would also like to thank Elisa Lorenzo, Giulio Meleleo and Piermarco Milione for their helpful discussions about their paper.

% ----------------------------------------------------------------
\bibliography{Distribution}
\bibliographystyle{amsplain}

\end{document}